%% file: HRAD.tex
\documentclass[preprint,pdftex,english]{imsart}
\usepackage{amsthm,amsmath,amssymb,amsfonts,epic,latexsym,multirow,ulem}

\usepackage[numbers]{natbib}
\RequirePackage{natbib}
\RequirePackage[colorlinks,citecolor=blue,urlcolor=blue]{hyperref}

\usepackage{wrapfig}
\usepackage[english]{babel}
 \usepackage[T1]{fontenc} 

 \usepackage{makeidx}
 \usepackage{fancyhdr}
\usepackage{epsf}
\usepackage[dvips]{epsfig}
\usepackage{subfig}
 \usepackage{graphicx}
\usepackage[nice]{nicefrac}
\usepackage{dsfont}

\usepackage{wasysym}
\usepackage{stmaryrd}
\usepackage{aeguill}
\usepackage{mathrsfs}
\usepackage{enumerate}

\startlocaldefs

\newtheorem{theo}{Theorem}[section]

\newtheorem{prop}[theo]{Proposition}
\newtheorem{coro}[theo]{Corollary}
\newtheorem{assu}{Assumption}

\newtheorem{lemm}[theo]{Lemma}
\newtheorem{rema}[theo]{Remark}

\def\R{\mathbb{R}}
\def \N{\mathbb{N}}
\def \Z{\mathbb{Z}}

\newcommand{\1}[1]{\mathbf{1}_{\set{#1}}}

\def \Pe{\mathbb{P}^{\mathcal{E}}}
\def \P{\mathbb{P}} % proba 
\def \Ee{\mathbb{E}^{\mathcal{E}}}
\def \E{\mathbb{E}} % esperance 
\def \F{{\cal F}} % tribu
\def \B{{\cal B}} % tribu borel

\newcommand{ \un }{\mathbf{1}}
\newcommand{\nuk}[1][k]{\nu_{#1}}

\newcommand{\egloi}{\stackrel{\mathcal{L}}{=}}

\newcommand{\cro}[1]{\left[ \left. #1 \right. \right]}
\newcommand{\GL}[1][\alpha]{G(#1)}

\newcommand{\Var}{\mathbb{V}\mbox{ar}}

\def\ud{\text{d}}

\def\T{\mathbb{T}}
\newcommand{\VM}{\overline{V}}
\newcommand{\Vm}{\underline{V}}

\newcommand{\red}[1]{{\color{red}{#1}}}

\newcommand{\ee}[1]{\e[\mathcal{E}]{#1}}
\newcommand{\pe}[1]{\p[\mathcal{E}]{#1}}

\newcommand{\Tn}[1][n]{T^{#1}}
\newcommand{\rang}[1][\alpha]{R^{\Tn}_{#1}}
\newcommand{\rangb}[2][n]{R^{\Tn[#1]}_{#2}}
\newcommand{\NT}[1][x]{N_{#1}^{\Tn}}
\newcommand{\NTi}[2]{N_{#1}^{\Tn[#2]}}

\endlocaldefs

\newcommand{\lemref}[1]{Lemma~\ref{lem:#1}} % Lemma tag equals ``lem''
\renewcommand{\T}{\mathbb{T}}

\newcommand{\paren}[1]{\left( \left. #1 \right. \right)} 
\newcommand{\set}[1]{\left\{ \left. #1 \right. \right\}}
\newcommand{\absj}[1]{\left\lvert #1 \right\rvert} %joli abs
\providecommand{\norm}[1]{\left \lVert #1 \right\rVert}
\renewcommand{\P}{\mathbb{P}}

\newcommand{\seq}[2][n]{\left(#2_{#1}\right)_{#1\in\N}}
\newcommand{\seqb}[2][n]{\left(#2_{#1}\right)_{#1\in\N^*}}

\newcommand{\est}[2][n]{\widehat{#2}_{#1}}

\newcommand{\NM}{N_{max}}

\newtheorem{postita}{Post-it}

\newcommand{\e}[2][]{\E^{#1}\!\cro{#2}}

\renewcommand{\P}{\mathbb{P}}
\newcommand{\p}[2][]{\P^{#1}\!\paren{#2}}
\newcommand{\bP}{\mathbf{P}}
\newcommand{\bp}[2][]{\bP^{#1}\!\paren{#2}}

\def\T{\mathbb{T}}

\newcommand{\Ldel}[1][\delta]{L_n^{#1}}

\begin{document}

\begin{frontmatter}

%%%%%%%%%
%%%%%%%%%%

\title{The heavy range of randomly biased walks on trees} 

%\runtitle{On the estimation of what we can !}

\author{\fnms{Pierre} \snm{Andreoletti}\ead[label=e1]{Pierre.Andreoletti@univ-orleans.fr}}
\address{
Institut Denis Poisson,   UMR CNRS 7013,
Universit\'e d'Orl\'eans, Orl\'eans, France. \printead{e1}}

\and 
\author{\fnms{Roland} \snm{Diel}\ead[label=e2]{diel@unice.fr}}
\address{Universit\'e C\^ote d'Azur, CNRS, LJAD, France. \printead{e2}}

\runauthor{Andreoletti, Diel}
\runtitle{The heavy range of randomly biased walks on trees}

\begin{abstract}
We focus on recurrent random walks in random environment (RWRE) on Galton-Watson trees. The range of these walks, that is the number of sites visited at some fixed time, has been studied in three different papers \cite{AndChen}, \cite{AidRap} and \cite{DeRap}. Here we study the heavy range: the number of edges visited at least $\alpha$ times for some integer $\alpha$. The asymptotic behavior of this process when $\alpha$ is a power of the number of steps of the walk is given for all the recurrent cases. It turns out that this heavy range plays a crucial role in the rate of convergence of an estimator of the environment from a single trajectory of the RWRE.  
\end{abstract}

% % 62M05 Markov processes: estimation
% % 62F12 Asymptotic properties of estimators
% % 60J25 Markov processes with continuous parameter
% % 60J27 Markov chains with continuous parameter
% % 60J35 Transition functions, generators and resolvents


\end{frontmatter}

%%%%%%%%
%%%%%%%%
\section{Introduction}

 In this paper, random walks in random environment (RWRE) on a supercritical Galton-Watson tree are considered.   We first focus on the number of edges frequently visited by the walk, this random variable is called {\it heavy range} in the paper. This can be seen in some sense as an extension of the works of \cite{AndChen}, generalized to the whole class of recurrent walks on trees. 
A second aim is to apply our control on the heavy range to a problem of non-parametric estimation for the distribution of the environment, extending this time the work of \cite{DieLer:2017} in the one-dimensional case.

Let us first give a precise definition of the process we are interested in. Consider a supercritical Galton-Watson tree $\T$ with offspring  distributed as a random variable $\nu$. % ($\E\cro{\nu}>1$ as $\T$ is supercritical). 
 In the paper, we adopt the following usual notations for tree-related quantities: the root  of $\T$ is denoted by $e$, for any $x \in \T$, $\nu_x$ denotes the number of descendants of $x$, %the variables $\nu_x$ are thus independent copies of $\nu$, 
 the parent of a vertex $x$ is denoted by $x^*$ and its children by $\set{x_i,\ 1\leq i\leq \nu_x}$. For technical reasons, we add to the root $e$, a parent $e^*$ which is not considered as a vertex of the tree. We also denote by $\llbracket x, y\rrbracket$ the sequence of vertices  realizing the unique shortest path between $x$ and $y$, by $|x|$ the generation of $x$ that is the length of the path $\llbracket e, x\rrbracket$ and we write $x<y$ (resp. $x\leq y$) when $y$ is a descendant of $x$ that is when $x$ is an element of $\llbracket e, y^*\rrbracket$ (resp. of $\llbracket e, y\rrbracket$). We also write $x\wedge y$ for the common ancestor of $x$ and $y$ belonging to the largest generation. Finally, we write $\T_n$ for the tree truncated at generation $n$. We then introduce a real-valued branching random walk indexed by $\T$: $\paren{V(x),x\in\T}$. We assume that $V(e)=0$ and we denote the increments of $V(x)$ by $\omega_x:=V(x)-V(x_*)$. For any generation $n$ and conditionally to $\mathcal{E}_n=\set{\T_n,(V(x),x\in\T_n)}$, the variables $\paren{\omega_{x_i},\ i\leq \nu_x}$ where $x$ is a vertex of $\T$ such that $|x|=n$ are assumed to be i.i.d. distributed as a random variable $\paren{\omega_i,i\leq\nu}$. 
 We denote by $\bP$ the distribution of $\mathcal{E}=\set{\T,\paren{V(x),x\in\T}}$; this variable is called the random environment of the walk.

For a given realization of $\mathcal{E}$, we consider the Markov chain $\seq{X}$ on $\T\cup\set{e^*}$ with transition probabilities defined by  the following relations: 
\begin{align*}
&\pe{X_{n+1}=e|X_n=e^*}=1\ ,\\
\forall x\in\T\setminus\set{e^*},\quad &\pe{X_{n+1}=x^*|X_n=x}=\frac{e^{-V(x)}}{e^{-V(x)}+\sum_{i=1}^{\nu_x}e^{-V(x_i)}}=\frac{1}{1+\sum_{i=1}^{\nu_x}e^{-\omega_{x_i}}}\ ,\\
	\forall j\leq\nu_x,\quad &\pe{X_{n+1}=x_j|X_n=x}=\frac{e^{-V(x_j)}}{e^{-V(x)}+\sum_{i=1}^{\nu_x} e^{-V(x_i)}}=\frac{e^{-\omega_{x_j}}}{1+\sum_{i=1}^{\nu_x} e^{-\omega_{x_i}}}\enspace.
\end{align*}
The measure $\Pe$ is usually referred to as the quenched distribution of the walk $\seq{X}$ in contrast to the annealed distribution $\P$ defined as the measure $\Pe$ integrated with respect to the law of $\mathcal{E}$: 
$$\p{\cdot}=\int\pe{\cdot}\bp{\ud \mathcal{E}}\ .$$
For $x\in\T\cup\set{e^*}$, we use the notation $\Pe_x$ for the conditional probability  $\Pe(\cdot | X_0=x)$. When there is no subscript, the walk is supposed to start at the root $e$. We finally introduce $\P^*$, the annealed probability conditioned on the survival set of the tree $\T$.

The walk $\seq{X}$, called biased random walk on a tree, was first introduced by R. Lyons (see \cite{Lyons} and \cite{Lyons2}). In our case where the bias is random, the first references go back to R. Lyons and R. Pemantle \cite{LyonPema} and M.V. Menshikov and D. Petritis \cite{MenPet}. Random walks in random environment on trees form a subclass of canonical models in the more general framework of random motions in random media that are widely used in physics. They are a natural extension of the one dimensional model, originally introduced in the works of \cite{Chernov}.  These models  have been intensively studied in  the last four decades, mostly   in   the    physics   and    probability   theory
literature.%, seminal papers are \cite{Solomon}, \cite{KesKozSpi} and \cite{Sinai} for the one-dimensional case and \cite{Lyons}, \cite{HuShi10}, \cite{HuShi10a} for walks on trees. 

 The behaviors of randomly biased walks on trees differ deeply from the behaviors of the RWRE in the one-dimensional case.  In particular there are several regimes for both recurrent and transient cases.  A complete classification for the recurrent cases is given by G. Faraud \cite{Faraud} (for the transient cases, see E. Aidekon \cite{Aidekon2008}). It can be determined from the fluctuations of log-Laplace transform $\psi(s):=\log \e{\sum_{|z|=1} e^{-sV(z)}}$ as resumed in Figure \ref{fig1}.
    \begin{figure}[ht]
\begin{center} 
{\scalebox{1.2} {\input{RecCrit.tex} }}
\caption{Recurrence criteria for $\seq{X}$ on Galton-Watson trees} \label{fig1}
\end{center}
\end{figure}
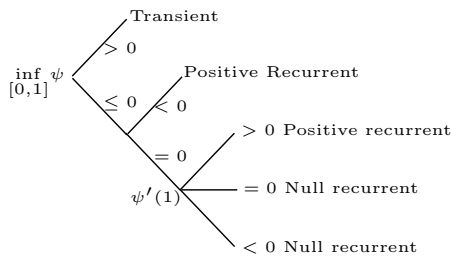

Here we focus on the recurrent cases that is to say when $\inf_{ s \in[0,1]} {\psi(s)} \leq 0$.  They present essentially three different asymptotic regimes which depend on the fluctuation of $\psi$. In the paper, we need a bunch of classical assumptions which are summarized here:
\newpage

\begin{assu}\label{ass1}$ $
	\begin{itemize}
	\item the Galton-Watson tree is supercritical: $\e{\sum_{|x|=1}1}=\e{\nu}\in(1,\infty)$.
	\item the log-Laplace transform $\psi$ is well defined on a neighborhood of $[0,1]$:
\begin{align}\label{hyp0}
\exists r_1,r_2>0,\ \forall s\in [-r_1,1+ r_2 ],\quad \psi(s)<\infty\ .
\end{align}
	\item  $\inf_{ s \in[0,1]} {\psi(s)} \leq 0$ (so $\seq{X}$ is recurrent).

	\item let $t_0:=\inf\set{s\geq0,\ \psi(s)=0}$ then:
\begin{align}\label{hyp1}
 \E\Big[\Big(\sum_{|z|=1}e^{-t_0V(z)}\Big)^2\Big] <+ \infty.
\end{align}
 and if $t_0=1$:
\begin{itemize}
\item if $\psi'(1) \geq 0$, there exists $\delta>0$ such that $\E[\nu^{1+ \delta}]<+\infty, $ 
\item if $\psi'(1)<0$, let
\begin{align}
\kappa:= \inf\{s>1, \psi(s)=0 \} \in (1,+\infty] \label{defkappa},
\end{align}
when $\kappa<\infty$,
\begin{align}
 \E\Big[\sum_{|z|=1} e^{-\kappa V(z)} \cdot \max(-V(z),0)\Big] <+ \infty\ . \label{hyppsineg}
\end{align}
\end{itemize}
\end{itemize}
\end{assu}
 
    Let us briefly describe the different recurrent cases (assuming the above conditions).
 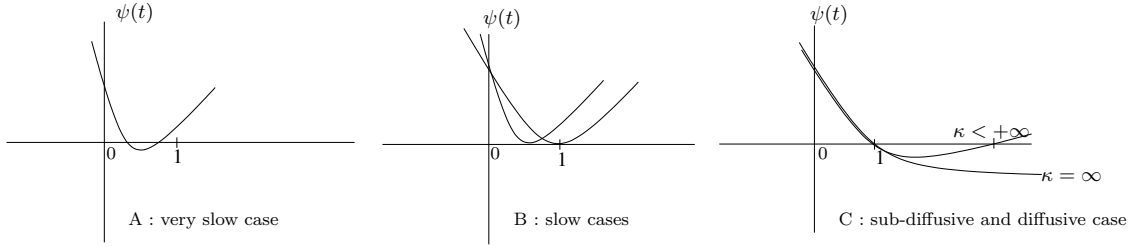
\begin{figure}[ht]
\begin{center} 
\scalebox{0.85} {\input{Psi.tex} }
\caption{Log-Laplace $\psi$ and corresponding behaviors of $\seq{X}$.} \label{fig2}
\end{center}
\end{figure}

First, when $\inf_{ s \in[0,1]} {\psi(s)} < 0$ (Figure \ref{fig2}: A), the walk is positive recurrent and the largest generation visited up to an instant $n$, $X_n^*:= \max_{k \leq n}|X_k|$ is of the order $\log n$ (\cite{HuShi10}), this is the slowest case, even slower than one dimensional Sinai's random walk (\cite{Sinai}).

If $\inf_{ s \in[0,1]} {\psi(s)} = 0$ (Figure \ref{fig2}: B and C), then different behaviors appear depending on the sign of $\psi'(1)$. When $\psi'(1) \geq 0$ (Figure \ref{fig2}: B), there are again two possible cases: if $\psi'(1)>0$ then the walk is positive recurrent whereas when $\psi'(1)=0$ the walk is null recurrent. However these two cases lead to the same asymptotic behavior for $X_n^*$ (up to some multiplicative constant) which is of the order $(\log n)^3$ (see \cite{HuShi10a} and \cite{HuShi10b}).   

Finally, when $\psi'(1) < 0$ (Figure \ref{fig2}: C), the walk is null recurrent and $X_n^*$ is of the order $n^{1- 1/ \min(\kappa,2)}$  where $\kappa$ is defined in \eqref{defkappa} (see \cite{HuShi10}, \cite{Faraud}, \cite{AidRap}, \cite{DeRap}).

\smallskip

In this paper we study the heavy range: the number of edges of the tree which are frequently visited by the walk. Precisely, for $x \in \T$, let $N_x^{n}$ be the number of times the walk $\seq{X}$ visits the edge $(x^*,x)$ before time $n$:
\[N_x^{n}:=\sum_{k=1}^{n} \1{X_{k-1}=x^*, X_{k}=x}\ . \] 
Then, for any $\alpha>0$ the  {\it heavy range} (of level $\alpha$) is the random variable
\begin{align} 
R^{n}_\alpha:=\sum_{x \in \T }  \1{ N_x^{n} \geq \alpha}\ .  \label{defR}  
\end{align}

Note that for $\alpha=1$, the above random variable is the usual range. In this case, the asymptotic behavior is known for all null recurrent cases, see  \cite{AndChen}, \cite{AidRap} and \cite{DeRap}.
The heavy range is then a natural generalization of the usual range which helps to understand how the random walk spreads on the tree. 
% In the following subsection %\ref{resproba}
% we describe first the asymptotic behavior of the heavy range at random time and then the behavior of $R^{n}_\alpha$ properly. Then in  subsection \ref{resstat}, we see how the heavy range controls the error rate of a natural estimator  of the distribution of the environment via a single trajectory of the RWRE.
% 
% \subsection{The behavior of heavy range $R_{n^{\theta}}^{n}$}\label{resproba}
\smallskip

Consider now the $n^{th}$ return to $e^*$: 
$$\Tn[0]:=0\quad \text{and}\quad \forall n\geq1,\ \Tn :=\inf\{ k>\Tn[n-1],\  X_k=e^* \}\ .$$
 The following theorem gives the behavior of the heavy range taken at time $\Tn$ in all recurrent cases. Note that in these cases, $\Tn$ is almost surely finite for every integer $n$.

\begin{theo}\label{thm1}
Assume \ref{ass1}. For any $\theta\geq0$, 
\begin{center}
$\frac{\log^+\rang[n^{\theta}]}{\log n}$ converges in $\P^*$-probability to a constant $\xi_\theta$ 	
\end{center}
 where $\log^+ x$ stands for $\log(\max(1,x))$. Moreover, when $\theta\geq1$, $\xi_\theta=0$ and, when $\theta<1$:
 \begin{itemize}
 \item if $\inf_{s \in [0,1]} \psi(s)<0$ or $\inf_{s \in [0,1]} \psi(s)=0$ with $\psi'(1) \geq 0$, then
 \begin{align} \label{SlowCases} 
 \xi_\theta=t_0(1-\theta).%,\textrm{ with recall } t_0= \inf\{s>0,\ \psi(s)=0\} \leq 1, 
 \end{align} 
\item if $\inf_{s \in [0,1]} \psi(s)=0$ with $\psi'(1) < 0$, then
\begin{align} 
\xi_\theta& = \kappa (1-\theta)\ \text{if }\ 1<\kappa   \leq 2, \label{SubDiffusiveCases} \\
\xi_\theta& =\max\Big(2 - \kappa \theta, 1-\theta\Big )\ \text{if }\ \kappa \in(2,\infty] \label{DiffusiveCases}.  
\end{align}
\end{itemize}
\end{theo}
%\begin{proof}
The proof of Theorem \ref{thm1} for $\theta<1$ is given in Section \ref{secSR}. The cases $\theta\geq1$ are then easily obtained using the results in the cases $\theta<1$ and the fact that, for any $n\geq1$, $\rang[n^{\theta}]$ is a non-increasing function of $\theta$.
%\end{proof}

Note that \eqref{SlowCases} corresponds to the slow cases, that is when $X_n$ behaves like a power of $\log n$. Also $t_0<1$ if and only if the random walk is positive recurrent whereas $t_0=1$ corresponds to the so called boundary case for the branching potential $(\psi(1)=\psi'(1)=0)$.

For the diffusive cases, that is when $\kappa>2$, the heavy range is of the order $n^{2 - {\kappa \theta}+o(1)}$,  larger than $n^{1-\theta}$ for $\theta\leq1/(\kappa-1)$. Conversely, for $\theta\geq1/(\kappa-1)$, that is to say if we are interested in sites often visited by the walk, the range is of the order $n^{1-\theta}$. We will see in the proof that this fact depends deeply where the edges sufficiently visited are localized on the tree. 

The asymptotics of the heavy range in deterministic time will be easier to interpret in terms of the behavior of the walk. They are described in Theorem \ref{Cortpsdet}. The theorem is proved using Theorem \ref{thm1} and a control on $\Tn$, it is done in Section \ref{sec3.3}.

The behavior of $\Tn$ has been studied in \cite{AndDeb2} and more precisely in \cite{HuShi15} and \cite{Hu2017}, for completeness we recall these results in the following remark; note that in these papers the results are given for  the $n^{\text{th}}$-return time to the root $e$ and not to $e^*$ but this does not change the rates of convergence. % the inverse of $\Tn $, called local time, but the result for $\Tn $  can be deduced by a change of variables.

\begin{rema} \label{localtimeroot}$ $
\begin{enumerate}
\item \label{un} If $\inf_{s \in [0,1]} \psi(s) = 0$, with $\psi'(1) = 0$ then in probability $\Tn/(n \log n)$ converges in probability to a positive limit (see \cite{HuShi15}, Proposition 2.5). 
\item \label{deux} If  $\inf_{s \in [0,1]} \psi(s) = 0$, $\psi'(1) < 0$ and  $\kappa \neq 2$, $\Tn/n^{\min(\kappa,2)}$ converges in law to a positive limit and if $\kappa=2$, $ (\Tn\log n)/n^2$ converges in law to a positive limit (see \cite{Hu2017}, Corollary 1.2).
\item \label{trois} Finally if $\psi'(1) > 0$ with $\inf_{s \in [0,1]} \psi(s) = 0$ or $\inf_{s \in [0,1]} \psi(s) < 0$ then $\Tn/n$ converges in probability to a positive limit (both are positive recurrent random walks).
\end{enumerate}
\end{rema}

\noindent The behavior of the heavy range in deterministic time is given in the following theorem:
%Thanks to the above remark and Theorem \ref{thm1}, we obtain the following corollary:
 \begin{theo} \label{Cortpsdet} 
  Assume \ref{ass1}. For any $\theta\in[0,1]$,
  \begin{center}
$\frac{\log^+ R^{n}_{n^{\theta}}}{\log n}$ converges in $\P^*$-probability to a constant $\tilde{\xi}_\theta$.
\end{center}
Moreover,
 \begin{itemize}
\item if $\inf_{s \in [0,1]} \psi(s) \leq 0$ with $\psi'(1) \geq 0$, then  $\tilde{\xi}_\theta=\xi_\theta=t_0(1-\theta)$,  
\item if $\inf_{s \in [0,1]} \psi(s)=0$ with $\psi'(1) < 0$ and 
 \begin{itemize}
 	\item if $1<\kappa\leq 2$, then $\tilde{\xi}_\theta=1-\kappa\theta$ for $\theta\leq 1/\kappa$ and 0 otherwise,
 	\item if $\kappa\in(2,\infty]$, then $\tilde{\xi}_\theta=\max(1-\kappa\theta\ ,\ 1/2-\theta)$ for $\theta\leq 1/2$ and 0 otherwise.
 \end{itemize}
\end{itemize}
 \end{theo}
 
Let us compare, for null recurrent cases, the heavy range $R^{n}_{n^{\theta}}$ for $\theta>0$ and the regular range, that is $ R^{n}_{1}$. It is proved in \cite{AndChen} that for the slow case, when $\psi(1)=\psi'(1)=0$, $ \frac{\log n}{n}R^n_1$ converges in probability to an explicit positive constant, whereas for the sub-diffusive and diffusive case, when $\psi(1)=0$, $\psi'(1)<0$, it is proved in \cite{AidRap} and \cite{DeRap} that the correct normalization for the convergence in probability of $R^n_1$ is simply $1/n$. So for the regular range we observe that sub-diffusive and diffusive cases spread more than in the slow case. For the heavy range the opposite appears, indeed when $\psi(1)= \psi'(1)=0$, $t_0=1$ and $\tilde{\xi}_\theta=1- \theta$ which is larger than $1- \kappa \theta$ for $\kappa>1$. This tells us that the environment of the slow null recurrent case creates much more vertices where the walk spends larger amount of time than in the other cases.
%\red{A developper un peu}

Here we obtain the first order for the asymptotic expansion of the heavy range for any recurrent cases and of course a natural question is now to obtain the correct normalization for this range. 

\subsection{Non-parametric estimation of the law of the environment}\label{resstat}

The study of the heavy range has been partly motivated by the statistical problem we present here. The statistical study of random processes in random environment has been overlooked in the literature  until recently, when new  biophysics experiments produced data that can be  modeled (at least in an ideal-case setup) by RWRE. For example in \citep{Monasson}, RWRE are used as a mathematical model of a mechanical denaturation of DNA. Consequently, a  new series of works appeared on statistical  procedures aiming  at estimating the distribution of the environment from a RWRE trajectory. 

We start with a brief resume of what has been done recently in the one dimensional discrete case. Recall that for  RWRE on $\Z$, the environment is a sequence of random variables $(\rho_x)_{x\in \Z}$.
When the environment is i.i.d., the estimation of the distribution of $\rho_0$ observing only the RWRE $\seq{X}$ is a natural problem. The first theoretical result appears  in \cite{AdeEnr} which treats the problem in a general settings with no quantitative purpose, then in \cite{andreo_sinai} and \cite{AndDie} particular cases are detailed. Recently more attention has been paid to the one dimensional case: the aim was to provide a parametric estimation of the distribution of the environment with the help of a single trajectory of random walk $\seq{X}$ (\cite{Comets_etal}, \cite{Falc_etal}, \cite{FaGlLo}, \cite{Comets_etal2}). When the environment is a Markov chain a parametric approach can also be found in \cite{AnLoMa}. The problem of non-parametric estimation has been studied in \cite{DieLer:2017}, the aim of the result we present below is to extend their studies to the more delicate case of randomly biased random walks on supercritical Galton-Watson trees.

Indeed, the next step after the model of RWRE for mechanical denaturation of DNA in \citep{Monasson}, is to construct a similar model for mechanical denaturation of RNA. But while DNA molecules are quite linear, RNA molecules present a more complicated geometry with secondary structure (see \cite{GeBuHw}). It seems then natural to describe RNA molecules with trees and to mimic its mechanical denaturation by a random walk in random environment on a tree. 

\smallskip
However, the estimation of the distribution of the environment for a random walk on a tree is much more complicated as we have to study the law of the transitions of a branching random walk and not only a real random variable. Therefore, we only consider here a very particular and simple case; this part of the paper has to be seen as a first step to understand how the methods used for the one-dimensional case can be generalized. 
 So, for all the statistical results in this part and in Section \ref{sec:moments}, we add to \ref{ass1} the following assumptions:
\begin{assu}\label{ass2}$ $
	\begin{itemize}
	\item the reproduction law of the Galton-Watson is bounded: $\exists K>0,\quad \p{\nu\leq K}=1$ .
	\item given the tree up to generation $n$ and the number of children $\nu_x$ of some $x\in\T$ such that $|x|=n$, the variables $\paren{\omega_{x_i}}_{1\leq i\leq \nu_x}$ are i.i.d. with the same distribution as some variable  $\omega$.
\end{itemize}
\end{assu}
Our aim is then to estimate the distribution of $\omega$ given the observation of a single trajectory of the walk $\seq{X}$ up to time $\Tn$. In particular we need $\Tn$ to be finite which is the case by recurrence of $\bf X$. Precisely, we give an estimation of the c.d.f $F$ of 
\begin{align}
\rho:=(1+e^{-\omega})^{-1} \label{defrho}
\end{align}
instead of  the one of $\omega$, but this is of course equivalent  as conversely $\omega=\log [{\rho}/(1-\rho)]$.

\begin{rema}
It is possible to relax the condition on $\nu$: if we suppose that the distribution is not bounded but only subgaussian, we still have the same rates of convergence in the following theorems. However the proof is more technical and for the sake of clarity we have chosen to present only the bounded case.	
\end{rema}

 We need to assume a regularity condition on $F$. For that we use $\gamma$-H\"older seminorms and spaces:
%\begin{itemize}
	%\item 
	for any $\gamma\in(0,1]$, the H\"older space $\mathcal{C}^\gamma$ is the set of continuous functions $f:[0,1]\to\R$ such that 
	$$\|f\|_\gamma=\sup_{u\neq v}\frac{|f(v)-f(u)|}{|v-u|^\gamma}<\infty$$ 
	and for $\gamma\in(1,2]$ the H\"older space $\mathcal{C}^\gamma$ is the set of continuously differentiable functions $f:[0,1]\to\R$ such that 
	$$\|f\|_\gamma=\|f'\|_\infty+\sup_{u\neq v}\frac{|f'(v)-f'(u)|}{|v-u|^{\gamma-1}}<\infty.$$

	\smallskip
	
Our first theorem gives rates of convergence for our estimator.
	
\begin{theo}\label{thm:ControlEstFModel}  Assume \ref{ass1}, \ref{ass2} and that the c.d.f. $F$ of $\rho$ is $\gamma$-H\"older  for some $\gamma\in(0,2]$.  There exists an estimator $\est{F}$ of $F$, function of the trajectory $\paren{X_{k}}_{0\leq k\leq \Tn}$,  such that 
	for any $\epsilon>0$,
		$$
	n^{r-\epsilon}\|\est{F}-F\|_\infty \text{\ tends to 0 in $\P^*$-probability where}
	$$
	
	\begin{enumerate}[i)]
\item \label{thun} $r=\frac{\gamma t_0}{\gamma+t_0}$ if $\inf_{s \in [0,1]} \psi(s) < 0$, $\inf_{s \in [0,1]} \psi(s) = 0$, with $\psi'(1) \geq 0$,

\item \label{thdeux} $r=\frac{\gamma\kappa}{\gamma+\kappa}$ if  $\inf_{s \in [0,1]} \psi(s) = 0$, $\psi'(1) < 0$ and  $\kappa \leq 2$,
\item \label{thtrois} $r=\frac{2\gamma}{\gamma+\kappa}$ if  $\inf_{s \in [0,1]} \psi(s) = 0$, $\psi'(1) < 0$ and  $ 2<\kappa\leq2+\gamma$,
\item \label{thquatre} $r=\frac{\gamma}{\gamma+1}$ if  $\inf_{s \in [0,1]} \psi(s) = 0$, $\psi'(1) < 0$ and  $ \kappa>2+\gamma$.
\end{enumerate}
\end{theo}

%\red{Je pense que $r$ n'est pas n\'ecessairement optimal car on a juste des majorations de l'erreur, je crois m\^eme que des simu numeriques dans le cas de $\Z$, montraient qu'il ne l'\'etait pas je ne suis pas s\^ur que cela vaille la peine de l'ecrire dans le papier par contre mais uniquement en r\'eponse au referee.}

In the previous theorem, the rate of convergence is given as a function of the parameter $n$. However, as the number of steps for the walk is the true sample size, it is more natural to give a rate depending on $\Tn$. This is done in the following corollary which is a direct consequence of Theorem \ref{thm:ControlEstFModel} and Remark \ref{localtimeroot}.

\begin{coro}
	\label{cor:ControlT}  Assume conditions of Theorem \ref{thm:ControlEstFModel}. The estimator $\est{F}$ is such that 
	for any $\epsilon>0$,
		$$
	\paren{\Tn}^{r-\epsilon}\|\est{F}-F\|_\infty \text{\ tends to 0 in $\P^*$-probability where}
	$$
	
	\begin{enumerate}[i)]
\item \label{corun} $r=\frac{\gamma t_0}{\gamma+t_0}\quad$ if $\inf_{s \in [0,1]} \psi(s) < 0$, $\inf_{s \in [0,1]} \psi(s) = 0$, with $\psi'(1) \geq 0$,

\item \label{cordeux} $r=\frac{\gamma}{\gamma+\kappa}\quad$ if  $\inf_{s \in [0,1]} \psi(s) = 0$, $\psi'(1) < 0$ and  $\kappa \leq 2+\gamma$,
\item \label{corquatre} $r=\frac{\gamma}{2(\gamma+1)}\quad$ if  $\inf_{s \in [0,1]} \psi(s) = 0$, $\psi'(1) < 0$ and  $ \kappa>2+\gamma$.
\end{enumerate}
\end{coro}

The corollary shows that our best rate is obtained in the limit case $t_0=1$. It seems to be the best compromise between the number of visited sites and the number of times most of the sites are visited.
Comparing our result on a tree to the recurrent one dimensional random walk of \cite{DieLer:2017}, we remark that the error rates are really different. Indeed, in their paper, the time of observation used instead of $\Tn$ is $\tau_n$, the first time the coordinate $n$ is reached by the walk and Corollary 2 in \cite{DieLer:2017} gives for the recurrent case an error rate of the order $\log \log \tau_n / \log \tau_n$. 
This is very large compared to what is obtained here in Corollary \ref{cor:ControlT}, so recurrent walks on the tree naturally yield a better rate for the error. The reason comes from the range of the different walks: for the one dimensional recurrent case very few coordinates are visited before time $n$, around $(\log n)^2$, whereas (see Theorem 1) for the walk on the tree the range is much larger (of the order of a power of $n$).

\smallskip

We now give a more explicit description of $\est{F}$. For this purpose, we introduce the following family of estimators
\begin{equation}\label{eq:estF}
\est{F}^\alpha(u):=\frac{1}{\rang\e{\nu}}\sum_{x\in\T}\psi^{\lfloor \alpha u\rfloor}_{\alpha}\paren{\NT[x^*],\NT[x]}
\end{equation}
where 
\begin{align*}
	\psi^{l}_{\alpha}(i,j)&:=\frac{\1{i\geq \alpha}}{\binom{i-1+j}{\alpha-1}}\sum_{k=0}^{l-1}\binom{i-1}{k}\binom{j}{\alpha-1-k}\enspace,
\end{align*}
using the conventions $0/0=0$ and $\binom{n}{k}=0$ if $0\leq n< k$. The logic behind the definition of the $\est{F}^\alpha$ will become clear in Section \ref{sec:moments}. The following theorem shows that the $(\est{F}^\alpha)_{n}$ are estimators of the c.d.f. $F$ with a risk bound depending on the heavy range $\rang$.

\begin{theo}\label{thm2.3}  Assume \ref{ass1}, \ref{ass2}  and that the c.d.f. $F$ of $\rho$ is in $\mathcal{C}^\gamma$ for some $\gamma\in(0,2]$.  
	Then for any integers $\alpha ,n\geq 1$, and any real $z>0$, we have
	\begin{align*}
	\p{\|\est{F}^{\alpha}-F\|_\infty\geq \frac{K}{\e{\nu}}\sqrt{\frac{z+\log \alpha +2\log \rang}{2\rang}}+\frac{2\|F\|_\gamma}{\alpha^{\gamma/2}}}\leq Ce^{-z}\enspace.
        \end{align*}
\end{theo}

Theorem \ref{thm2.3} shows that the random part of the error rate of our estimator is a function of the heavy range introduced in the previous subsection. To obtain the random optimal $\est{F}^{\alpha}$ for the estimation of $F$, a compromise must be done between considering sites which have been sufficiently visited, that is choosing a large $\alpha$, and considering a sufficiently large number of sites, that is choosing $\alpha$ small enough. 
Remark that, in this theorem, we work with probability $\P$ and not $\P^*$. Theorem \ref{thm:ControlEstFModel} and \ref{thm2.3} will be proved in Section \ref{sec:moments}. %But we can directly explain how it implies \ref{thm:ControlEstFModel}.

\subsection{Overview of the proofs}

  The proof of Theorem \ref{thm1} is divided into two parts : an upper bound given in Proposition \ref{proprange}  and  lower bounds given in Propositions \ref{proprange0} and  \ref{proprange0bis}. Propositions \ref{proprange0} explains the contribution of the first generations of the tree, while \ref{proprange0bis} describes the contribution of larger generations. 
For all these bounds, arguments are adapted to the hypothesis done on the environment. More specifically important differences appear whether the largest generation visited by the walk before $n$ returns to $e^*$, $ X^*_{\Tn}$, is small that is typically a power of $\log n$  or larger, that is typically a power of $n$. These behaviors are resumed in Lemma \ref{asympbeh}.  

The upper bound of Proposition \ref{proprange} is  obtained by an estimation of the mean of the heavy range $\rang[n^\theta]$ and by Markov's inequality. The most technical part  is the study of the case $X^*_{\Tn}  \sim n^p$ (Case 2 in the proof of Proposition \ref{proprange}): indeed for this case, some sites with  large potential which are far from the root are of major importance. It appears that these sites are essentially visited during a single excursion from $e^*$ to $e^*$ (contribution $\sum_{2,2}$ in the proof) but that the walk is trapped a long time in their neighborhood. Then thanks the many-to-one Lemma  (Lemma \ref{manytoone} in Section \ref{secEnv}) the central issue is to study a random walk $\seq{S}$ on $\Z$ and especially a related random variable $H^S_.$ for which the tail distribution function is estimated in Lemma \ref{compHS} in Section \ref{secEnv}.

For the lower bound, which is more delicate to obtain, we decompose the proof into two Propositions. Proposition \ref{proprange0} is dedicated to the contribution of the first generations of the tree, that is sites $x$ such that $|x|\leq (\log n)^3$. It works for every cases (slow and fast) but is not always optimal for the fast ones. It is based on the idea that any sites $x$ close to the root such that $\max_{u \leq x} V(u) \leq (1- \theta- \delta) \log n$ for a small $\delta>0$ satisfy, with a high probability  $\NT\geq n^{\theta}$ (see \eqref{eqAn}). This fact is obtained by a precise estimate of the Laplace transform of $\NTi{x}{1}$. Then it remains to show that the cardinal of the set $\{x \in \T, \ |x| \leq (\log n)^3, \overline{V}(x) \leq (1- \theta- \delta) \log n \}$ is larger than $n^{t_0(1- \theta- 2\delta)}$. This part is the object of Proposition \ref{propenv1}. \\
Proposition \ref{proprange0bis} completes the estimation of Proposition  \ref{proprange0} for the fastest cases. Indeed in these cases, the random walk can visit generations of order a power of $n$ before the $n^\text{th}$ return to $e^*$ and can be trapped a long time in some sites of these generations. More precisely in this proposition, we consider the vertices on the tree such that the variable $ H_x:=\sum_{y\leq x}e^{V(y)-V(x)}$  is of order $n^{\theta}$. The reason for that is the same as above: whenever the walk touches such an $x$, its local time at this vertex will be, with a large probability, larger than $n^{\theta}$. \\
A first step in the proof is %, like for the upper bound, to consider only vertices $x$ with a high level of potential (typically $V(x) \geq 4 \log n$) indeed we know from the upper bound that only these vertices count, and more than that it allows
 to introduce time independence, under quenched measure $\Pe$, as those $x$ are visited only during a single excursion between two returns to $e^*$ (cf. Lemma \ref{eqindep}). This independence allows to apply Tchebytchev inequality : see Lemma \ref{lemtcheb}. From this lemma, it appears that two random variables which depend only on the environment have to be controlled. And thanks to concentration Lemma \ref{lemconcen}, a control depending only on their means and variances can be obtained. These quantities are computed in Lemma \ref{lemGn} and Lemma \ref{binfmoy}.% that  First random variable $ \sum_{L_n^{\delta}}e^{-V(x)} \un_{x \in \Gamma_n}$, where $L_n^{\delta}$ is a small perturbation of $L_n$ and $\Gamma_n$ a set of good environment is estimated in Lemma \ref{binfmoy}. The second one  $\sum_{|x|,|y|\leq \Ldel}H_{x\wedge y}e^{-V(x)-V(y)+V(x\wedge y)}\1{x,y\in\Gamma_n}$ which comes from the variance of Tchebytchev inequality, is controlled in Lemma \ref{lemGn} by the computation of its mean. \\

% As we have seen above, in the proof of Theorem \ref{thm1} there is two important results which deal only with the environment  : Proposition \eqref{propenv1} and Lemma \ref{binfmoy}. The proof of both of these results (Section \ref{secEnv}) is based on the same idea : the fluctuations of $V$ in the first generations of the tree have no consequences on its fluctuations for large generations. So we use this aspect to make appear independence by cutting the tree at early generation and then use a concentration lemma by considering resulting independent subtrees (Lemme \ref{lemconcen}). 

\smallskip

  The statistical results are obtained by the estimation of the moments of $\rho=(1+e^{-\omega})^{-1}$. In Proposition \ref{thm:BorneRisqueGW}, we give a risk bound for an estimator of the moments  $m^{\alpha,\beta}:=\e{\rho^{\alpha}(1- \rho)^{\beta}}$. Remark that the randomness of this bound depends only on the heavy range.
 Then, thanks to the H\"older regularity of $F$, this function can be approximated by the family of functions $u\to F^\alpha(u)=\sum_{k=0}^{\lfloor\alpha u\rfloor-1}\binom{\alpha-1}{k}m^{k,\alpha-1-k}$ (see Lemma \ref{lem:biasF}). We can therefore use (a slight variation of) the estimators introduced in Proposition \ref{thm:BorneRisqueGW} to construct estimators of $F^\alpha$ and give a risk bound, this is done in Lemma \ref{lem:BorneRisqueF}. Theorem \ref{thm2.3} follows directly while the last section explains how to obtain Theorem \ref{thm:ControlEstFModel}.

% 
%Computation of the likelihood makes appear that the statistic is given by the sequence $(N_x^{(n)}, x \in \T)$ (see the beginning of Section \ref{sec:moments}), and it turns out that moments of $\rho$ and more particularly $m^{\alpha,\beta}:=\E(\rho^{\alpha}(1- \rho)^{\beta})$ can be expressed as the mean of a certain random variable $\hat m_n^{\alpha, \beta}$ (ou estimator) only function of $(N_x^{(n)}, x \in \T)$ and the heavy range. The heavy range coming as a normalization for this estimator, so the main step is to obtain the convergence of $\hat m_n^{\alpha, \beta}$ to $m^{\alpha,\beta}$.

\smallskip

The rest of the paper is organized as follows. In Section \ref{secEnv} we present tools related to the environment together with estimates for potential $V$ that will be used in Section \ref{secSR}. More specifically we study the number of vertices with low potential which have a great importance for the slow cases. Section \ref{secSR} is the heart of the paper, we study the heavy range and prove Theorem \ref{thm1}. This section is decomposed into two subsections the first one deals with an upper bound for the heavy range and the second one with a lower bound. Note that the proof of the lower bound is more technical and need in particular to estimate the fluctuations of the environment, these estimates are more complex for the sub-diffusive and diffusive cases than for the slow cases. Then, in Section \ref{sec:moments}, we prove Theorem \ref{thm:ControlEstFModel} on the non-parametric estimation problem of the environment.

Finally, throughout the paper, the letter $C$ stands for a universal constant which value can change from line to line.

\section{Preliminaries on the environment}\label{secEnv}

In this section we present results related to the environment which are used in Section \ref{secSR}. % to prove Theorem \ref{thm1}. %\red{Like in the whole paper, we assume in this section that Assumptions \ref{ass1} are valid.  In a first subsection we give useful technical lemmata: the first one recalls (without proof) a usual technical result known as many-to-one Formula, the second one is a tail estimate, originally proved by H. Kesten \cite{Kesten4}, of a key random variable which is related to the edge local-time. % via this many-to-one formula. \\the third one is a bound on the maximum of the potential, and the last one is a concentration result. The third lemma is actually used to obtain independence between subtrees (by making cut on the tree) this is crucial to apply the forth lemma. Finally, in a second subsection, we state and prove a result concerning  the number of vertices with low potential $V$.}
 Like in the whole paper, we assume in this section that conditions \ref{ass1} are valid.  In a first subsection we give useful technical lemmata and in a second one, we state and prove a result concerning  the number of vertices with low potential $V$.
 
 Let us begin with some notations specific to the environment. For $u\in\T$, we denote by $V_u$ the environment centered at $u$: for any $x\geq u$, $V_u(x):=V(x)-V(u)$. We also denote the maximum of $V_u$ between $u$ and $x$ by $\overline{V_u}$ and the minimum of $V$ between $e$ and $x$ by $\underline{V_u}$:
\begin{align}
\overline{V_u}(x):=\max_{u\leq y\leq x} V(y)\quad \text{ and }\quad \underline{V_u}(x):=\min_{u\leq y\leq x} V(y). \label{VmVM}
\end{align}
We write $\VM$ for $\overline{V_e}$ and $\Vm$ for $\underline{V_e}$. Remark that if $u$ and $v$ are different vertices of the same generation $\ell$, then, given $\T_\ell$, $\paren{V_u(x)}_{x\geq u}$ and $\paren{V_v(x)}_{x\geq v}$ are i.i.d. and distributed as $V$ under $\P$.

\subsection{Technical estimates}
We first recall the many-to-one Formula (see \cite{Shi2015} Chapter 1, and \cite{HuShi10b} equation 2.1). It will be used several times in the paper to compute different expectations related to the environment.
\begin{lemm}[Many-to-one Formula]\label{manytoone}
For any integer $m$ and any $t>0$,
	\begin{align*}
\E\Big[\sum_{|x|=m}f(V(y),e \leq y \leq x)\Big]= \e{e^{tS_m+\psi(t) m }f(S_i,0 \leq i \leq m)}.  
\end{align*}
where $\seq{S}$ is the random walk starting at 0 and such that the increments $\paren{S_{n+1}-S_n}_{n\in\N}$ are i.i.d. and for any measurable function $h:\R^m\to [0,\infty)$, 
$$
\e{h(S_1)}=e^{-\psi(t)}\e{\sum_{|x|=1}e^{-tV(x)}h(V(x))}.
%\frac{\e{\sum_{|x|=1}e^{-tV(x)}h(V(x))}}{\e{\sum_{|x|=1}e^{-tV(x)}}}=e^{-\psi(t)}\e{\sum_{|x|=1}e^{-tV(x)}h(V(x))}\ .
$$
\end{lemm}

The following lemma deals with a key random variable which appears in the study of the heavy range (via the edge local time of random walk $\seq{X}$, see below \eqref{sig21}) after that many-to-one Formula is applied.

\begin{lemm}\label{compHS}Assume $\psi(1)=0$ and $\psi'(1)<0$ and that the parameter $\kappa$ defined in \eqref{defkappa} is finite. Let us consider the random walk $\seq{S}$ of the many-to-one Lemma with $t=1$ and define for any $\ell\geq0$, the random variable: 
\begin{align}%\label{defHS}
	H_\ell^S:=\sum_{k=0}^\ell e^{S_k-S_\ell}\ .
\end{align}
 
Then, we can find three constants $C\geq c$ and $A>0$ such that
$$
\forall m\geq1,\ \forall \ell\in\N,\quad \p{H_\ell^S\geq m}\leq \frac{C}{m^{\kappa-1}},
$$
and 
$$
\forall m\geq1,\ \forall \ell\geq A\log m,\quad \p{H_\ell^S\geq m}\geq \frac{c}{m^{\kappa-1}}\ .
$$
\end{lemm}

\begin{proof}
 By definition of $S_1$,
 $\e{- {S}_1}=\psi'(1)<0\text{ and }\e{e^{ -(\kappa-1) {S}_1}}=e^{\psi(\kappa)}=1.$
 According to Lemma 1 in \cite{KesKozSpi} for the lattice case and to Theorem 2 in \cite{Grince76} for the non-lattice case, $\sum_{j=0}^{+ \infty} e^{-S_j}$ is $\P$-a.s. finite and there are some constants $c,C$, such that for $m\geq1$,
$$\frac{c}{m^{\kappa-1}}\leq \P\Big(\sum_{j=0}^{+ \infty} e^{-S_j}\geq m\Big) \leq \frac{C}{m^{\kappa-1}}\ .$$
As for any $\ell\in\N$, $H^S_\ell \egloi\sum_{k=0}^\ell e^{-S_k}$ where $\egloi$ stands for the equality in law, this leads directly to the upper bound of the lemma. For the lower bound, we remark that 
$$
\p{H_\ell ^S\geq m}\geq \P\Big(\sum_{k=0}^\infty e^{-S_k}\geq \frac32 m\Big)-\P\Big(\sum_{k=\ell+1}^\infty e^{-S_k} > \frac m2\Big).
$$
And we only have to prove that for $\ell$ large enough, $\p{\sum_{k=\ell+1}^\infty e^{-S_k}> m/2}$ is negligible compared to $1/m^{\kappa-1}$. Indeed, consider the event $A_\ell:=\{ \forall k \geq \ell+1, e^{-S_k} \leq \frac{m}{2k^2}\}$. On $A_\ell$, 
\begin{align*}
\sum_{k={\ell+1}}^{\infty} e^{-S_k} \leq \frac{m}2 \sum_{k={\ell+1}}^{ \infty} \frac1{k^{2}} \leq \frac{m}{2\ell}.
\end{align*}
Therefore, Markov inequality yields
\begin{align*}
& \P\Big(\sum_{k=\ell+1}^\infty e^{-S_k} >  m/2 \Big) \leq \p{\overline{A}_\ell} = \P\Big( \bigcup_{k \geq \ell+1} \set{e^{-S_k} > \frac{m}{2k^2}}\Big)\\
\leq& \sum_{k \geq \ell+1} \P\Big( e^{-(\kappa-1) S_k/2} > \paren{\frac{m}{2k^2}}^{(\kappa-1)/2}\Big) \\
  \leq& \paren{\frac{2}{m}}^{(\kappa-1)/2}\sum_{k \geq \ell+1} k^{\kappa-1}e^{ k \psi(\frac{\kappa+1}{2})} \leq C m^{-(\kappa-1)/2}  \ell^{\kappa-1}e^{ \ell \psi(\frac{\kappa+1}{2})}, 
\end{align*}
As $\psi(\frac{\kappa+1}{2})<0$, we can find a constant $A$ such that for $\ell\geq A \log m$,  the above expression is $o(m^{-(\kappa-1)})$. This concludes the proof of the lemma.
\end{proof}

We give now a  control of the maximum of $V$ at the very first generations of the tree. Recall that $\P^*$ is the probability $\P$ conditioned on the survival of the tree.
\begin{lemm}\label{ctrlptgen}
	%Assume \ref{ass1}. 
	For any $\delta>0$, we can find two positive constants $\epsilon,b_1$ such that for $n$ large enough,
\begin{align*}
\P^*\Big(\max_{|u|\leq \epsilon \log n} |V(u)| \geq \delta \log n\Big) \leq n^{-b_1}\ . 
\end{align*}
\end{lemm}
Note that the optimal bound, that is to say the one that leads to the almost-sure behavior of $\max_{|u| \leq  \epsilon \log n} {V}(u)$, is well known (see for example \cite{AndDeb1}) but its exact value has no importance here.

\begin{proof}
	Recall the definition of $r_1$ in \eqref{hyp0}, for any $\epsilon>0$ and $a>0$,
\begin{align*}
\P\Big(\max_{|u|\leq \epsilon \log n} V(u) \geq a \epsilon\log n\Big) &\leq \sum_{j \leq \epsilon \log n} \E\Big[ \sum_{|v|=j} \1{V(v) \geq a \epsilon\log n}\Big]\leq \sum_{j \leq \epsilon \log n}\E\Big[\sum_{|v|=j} e^{r_1V(v)-r_1a \epsilon\log n}\Big]\\
&\leq\sum_{j \leq \epsilon \log n} e^{-r_1 a \epsilon\log n+j\psi(-r_1)}\leq C e^{- \epsilon\log n(r_1 a-  \psi(-r_1) )}
\end{align*}
for some constant $C>0$. Then choosing $a$ large enough so that $r_1 a> \psi(-r_1)$  and taking $\epsilon=\delta/a$, we obtain the bound for $\max V$. 
For $\min V$, the proof is the same except we work with $t_0$ instead of $r_1$:
\begin{align*}
&\P\Big(\min_{|u|\leq \epsilon \log n} V(u) \leq -a\epsilon \log n \Big) \leq \sum_{j \leq \epsilon \log n} \E \Big [ \sum_{|v|=j} \1{V(v) \leq -a \epsilon\log n} \Big ]\\
\leq& \sum_{j \leq \epsilon \log n}\E \Big [\sum_{|v|=j} e^{-t_0V(v)-t_0a \epsilon\log n} \Big]\leq\sum_{j \leq \epsilon \log n} e^{-t_0 a \epsilon\log n}\leq C n^{-a\epsilon t_0}\log n .
\end{align*}
 As the non-extinction probability is positive, there is a constant $C>0$ such that for any $n\geq1$, 
 $$\P^*\Big(\max_{|u|\leq \epsilon \log n} |V(u)| \geq \delta \log n\Big)\leq C\P\Big(\max_{|u|\leq \epsilon \log n} |V(u)| \geq \delta \log n\Big).$$
 which concludes the proof.
\end{proof}

The above lemma is used in the following section to obtain independence between the different branches of the tree. It will be used together with the following concentration lemma.

\begin{lemm} \label{lemconcen}%Assume \ref{ass1}. 
Consider two integers $\ell$ and $L$ and $L+1$  sets $A_1\in\B(\R),\dots, A_{L+1}\in\B(\R^{L+1})$. Define, for any $u\in\T$ such that $|u|=\ell$, the variable
$$
Z_u:=\sum_{u\leq x,|x|\leq L+\ell}e^{-t_0V_u(x)}\1{ \paren{V_u(y),\ u\leq y\leq x}\in A_{|x|-\ell+1}}\ .
$$
 There exists a constant $a>1$ depending only on the reproduction distribution $\nu$ such that, for $\ell$ large enough,
\begin{align*}
\P^*\Big( \sum_{|u|= \ell}  Z_{u}< \e{Z}\Big)
\leq  \frac{\e{Z^2}}{\e{Z}^2 }a^{-\ell}
\end{align*}
where 
$$Z=\sum_{|x|\leq L}e^{-t_0V(x)}\1{\paren{V(y),\ y\leq x}\in A_{|x|}}\ .$$
\end{lemm}

\begin{proof}
We first work with the conditional probability $\P(\cdot|\T_{\ell})$ so that random variables $(Z^n_{u},|u|=\ell)$ are independent and identically distributed as $Z$ under $\P$. Denote by $D_\ell$ the number of vertices at generation $\ell$. Tchebychev's inequality gives: 
\begin{align*}
 \P\Big( \Big|\sum_{|u|= \ell}  Z_{u}-D_\ell\e{Z}\Big|> \frac{D_\ell}{2}\e{Z}\Big|\T_{\ell}\Big)
\leq \frac{4}{D_\ell}\frac{\e{Z^2}}{\e{Z}^2}\ .
\end{align*}
 Now, we have to control $D_\ell$. In the B\"ottcher case, i.e. $\nu(\set{0,1})=0$, $D_\ell\geq 2^\ell$ $\P^*$-a.s.($=\P$-a.s.) and the result of the lemma follows. 
 
 In the Schr\"oder case, i.e. $\nu(\set{0,1})>0$, Corollary 5 in \cite{Fleiwach} and Theorem 1 in \cite{Dubuc71} (see also Theorem 4 in \cite{BigBin93}) tell us that we can find two constants $c_1,c_2>1$ such that for $\ell$ large enough,
 $\P^*(D_\ell\leq c_1^\ell)\leq c_2^{-\ell}.$
 %For any event $A$,  $\P^*(A)=\E^*(\P( A|\T_{\ell}))=\E^*(\P( A|\T_{\ell}) \un_{D_\ell > c_1^\ell}) + \E^*(\P( A|\T_{\ell}) \un_{D_\ell\leq c_1^\ell})$ 
 Therefore, in this case denoting by $P_{ne}>0$ the non-extinction probability 
 \begin{align*}
\P^*\Big( \sum_{|u|= \ell}  Z_{u}< \e{Z}\Big) &\leq \frac{\P\Big( \sum_{|u|= \ell}  Z_{u}< \e{Z},\ D_\ell > c_1^\ell\Big)}{P_{ne}}+c_2^{-\ell}\\
&\leq \frac{\E\Big[\P\Big(\sum_{|u|= \ell}  Z_{u}< D_\ell \e{Z}/2\Big|\T_\ell\Big) \un_{D_\ell > c_1^\ell}\Big]}{P_{ne}}+c_2^{-\ell}\\
&\leq \frac{4}{P_{ne}}\frac{\e{Z^2}}{\e{Z}^2}c_1^{-\ell}+c_2^{-\ell}\leq \frac{4}{P_{ne}}\frac{\e{Z^2}}{\e{Z}^2 }(c_1^{-\ell}+c_2^{-\ell}),
\end{align*}
and we only have to take $a<c_1\wedge c_2$ to conclude the proof.
\end{proof}

\subsection{Number of vertices with low potential \label{SecLP}}
Fix some constant $c>0$. In this subsection we are interested in a lower bound for the random variable   \[ \sum_{|x|  \leq (\log n)^3 } \1{\overline{V}(x) \leq c \log n}.  \] 
where $\VM$ is defined in \eqref{VmVM}. It counts the number of sites $x$ with generation smaller than $(\log n)^3$ such that the potential in the path from the root to $x$ remains below $c \log n$. We will see in the next section that this random variable is naturally related to the heavy range.

The proof we propose here is based on Lemma \ref{ctrlptgen} which implies that the very first generations of the tree have no important impact on the value of $\overline V$. This point can be used to obtain independence and then apply Lemma \ref{lemconcen}.

\begin{prop} \label{propenv1}Assume \ref{ass1}. Let $c>0$, For any $ \delta \in(0,1\wedge c)$, there exists $a>0$ such that for large $n$,
\begin{align*} 
{\P^*}\Big( \sum_{|x|\leq (\log n)^3} \1{\overline V(x) \leq c \log n}  \leq n^{t_0 (c-\delta) }\Big) \leq n^{- a}\ .
\end{align*}
\end{prop}

\begin{proof}

We first use Lemma \ref{ctrlptgen} to control $\overline{V}(u)$ on the very first generations of the tree:  
 we can choose $\epsilon>0$ and $b_1>0$ such that for $n$ large enough $\P^*\big(\overline{\mathcal{A}}_n\big) \leq n^{-b_1}$ with $\mathcal{A}_n:= \{\max_{|u|= \epsilon_n} \overline{V}(u) < \frac\delta4\log n\}$ and $\epsilon_n=\lfloor \epsilon\log n\rfloor$. %According to Lemma \ref{ctrlptgen}, there exist we can choose $\epsilon$ and $b>0$ such that  $n\geq1$, $\P^*\big(\overline{\mathcal{A}}_n\big) \leq n^{-c_1}$.

The strategy is then similar for the different cases (depending on $\psi$), the only difference is the generation we have to work with. %We define for $u\leq x$ the re-centered potential $V_u(x):=V(x)-V(u)$ and for a well chosen $\ell_n = \ell_n( \psi)$ and 
Let us take some constant $B \geq 0$ and  a sequence of integers $\seq{\ell}$ such that, for $n$ large enough, $\ell_n+\epsilon_n$ is smaller than $(\log n)^3$; the exact choice of the sequence depends on $\psi$ and will be explicitly given later. Consider now the collection of random variables: 
$$ \forall n\geq 1,\ \forall u\in\T,\quad Z_u^n:=\sum_{\substack{ x\geq u \\ |x|=\ell_n+\epsilon_n}} e^{-t_0 V_u(x)} \1{(c-\delta/2)\log n\leq V_u(x),\overline{V}_u(x) \leq (c-\delta/4)\log n,\  \underline{V}_u(x)\geq -B}\ .$$
For $n$ large enough, on the event $\mathcal{A}_n$,
\begin{align*}
 & \sum_{|x|\leq (\log n)^3}  \1{\overline{V}(x) \leq c\log n }  \geq \sum_{|u|=\epsilon_n} \sum_{\substack{ x\geq u \\ |x|=\ell_n+\epsilon_n}}  \1{\overline{V}_u(x) \leq (c-\delta/4)\log n}\geq n^{t_0(c-\delta/2)}  \sum_{|u|=\epsilon_n}  Z_u^n.
\end{align*}
Hence, concentration Lemma \ref{lemconcen} implies that there exists a constant $b_2>0$ such that for $n$ large enough,
\begin{align}\label{eqconc}
\P^* \Big( \sum_{|x|\leq (\log n)^3}  \1{\overline{V}(x) \leq c\log n }\leq n^{t_0(c-\delta/2)} \e{Z^n} \Big)
\leq  \frac{\e{\paren{Z^n}^2}}{ \e{Z^n}^2}n^{-b_2}+\p[*]{\overline{\mathcal{A}}_n}
\end{align}
where
$$\forall n\geq 1,\quad Z^n:=\sum_{|x|=\ell_n} e^{-t_0 V(x)} \1{(c-\delta/2)\log n\leq V(x),\overline{V}(x) \leq (c-\delta/4)\log n,\  \underline{V}(x)\geq -B}.$$
So we only have to control $\e{Z^n}$ and $\e{\paren{Z^n}^2}$ to prove the proposition. 

\smallskip

\noindent $\bullet$ First for the mean $\e{Z^n}$ the proof is different depending on the value of $\psi'(t_0)$:

 \textbf{Case 1: $\psi'(t_0)<0$}. Recall that in this case either $t_0=1$ and $\psi(1)=\inf_{s \in [0,1]} \psi(s)=0$, or  $t_0<1$ and $\inf_{s \in [0,1]} \psi(s)<0$. We choose $\ell_n=\lfloor\ell\log n\rfloor$ with $\ell=(c-3\delta/8)/|\psi'(t_0)|$. Many-to-one formula of Lemma \ref{manytoone} with $t=t_0$ gives
\begin{align*}
\e{Z^n}& = \p{(c-\delta/2)\log n\leq S_{\ell_n},\  \overline{S}_{\ell_n} \leq (c-\delta/4)\log n,\ \underline{S}_{\ell_n}\geq -B} \\
& \geq  1-\p{(c-\delta/2)\log n> S_{\ell_n}}-\p{ \overline{ S}_{\ell_n} >  (c-\delta/4)\log n}-\p{\underline{S}_{\ell_n}< -B}.
\end{align*}
For any $B \geq 0$ and $\lambda>0$ such that $\psi(t_0+ \lambda)<0$, Markov inequality yields 
\begin{align*}
\p{\underline{S}_{\ell_n}<-B}&\leq \sum_{k \leq \ell_n} \p{S_{k}<-B}\leq \sum_{k \leq \ell_n} e^{- \lambda B} e^{ \psi(t_0+ \lambda) k}\leq\frac{e^{-\lambda B}}{1-e^{ \psi(t_0+ \lambda)}}\ ,
\end{align*}
 so taking $B$ large enough, for any $n$,  $\P(\underline{S}_{\ell_n}<-B)\leq C <1$. And, as the process $S$ is the sum of i.i.d. random variables with mean $|\psi'(t_0)|$ and $c-\delta/2 <|\psi'(t_0)| \ell < c-\delta/4$, using exponential Markov inequality we can also prove that for some constant $b_3>0$,
$$\p{(c-\delta/2)\log n> S_{\ell_n}} \leq n^{-b_3}\quad\text{ and }\quad\p{ \overline{ S}_{\ell_n} >  (c-\delta/4)\log n} \leq n^{-b_3}\ .$$   
Combining these bounds, we obtain that for any $n\geq1$,
\begin{align}
\e{Z^n} \geq \frac1C>0\ . \label{38b}
\end{align}

  \textbf{Case 2: $\psi'(t_0) = 0$}. In this case either $t_0<1$ or $\psi(1)=\psi'(1)=0$.  We can basically use the same method as in the case $\psi'(t_0)<0$ except that important generations (see \cite{AndDeb2}, \cite{AndChen}) are now the ones of the order $(\log n)^2$. So we take this time $\ell_n= \lfloor( \ell \log n)^2\rfloor$ for some $\ell>0$ and we can choose $B=0$.  Many-to-one Lemma yields like above
\begin{align*}
\e{Z^n} 
 & = \p{(c-\delta/2)\log n\leq S_{\ell_n},\  \overline{S}_{\ell_n} \leq (c-\delta/4)\log n | \underline{S}_{\ell_n}\geq0} \p{ \underline{S}_{\ell_n}\geq0}.
\end{align*}
 We know (see \cite{AidShi}, {equation (2.8)}) that the limit  
 $\lim_{n \rightarrow + \infty} {\ell_n}^{1/2} \P (  \underline{S}_{\ell_n}   \geq  0)= d>0 $ exists, also 
 by invariance principle (see \cite{Bolth}), 
\begin{align*}
& \lim_{n\to\infty}\p{(c-\delta/2)\log n\leq S_{\ell_n},\  \overline{S}_{\ell_n} \leq (c-\delta/4)\log n\ \big|\ \underline{S}_{\ell_n}\geq0} \\
=& \lim_{n\to\infty}\p{(c-\delta/2) (\sigma \ell )^{-1}  \leq S_{\ell_n}/ \sigma \sqrt {\ell_n},\  \overline{S}_{\ell_n} / \sigma \sqrt{\ell_n} \leq (c-\delta/4) (\sigma \ell )^{-1}\ \big|\ \underline{S}_{\ell_n}\geq0} \\ 
=&\ \p{(c-\delta/2) (\sigma \ell )^{-1}  \leq \mathcal{M},\  \overline{\mathcal{M}} \leq (c-\delta/4) (\sigma \ell )^{-1} } \geq \frac1C >0.
\end{align*} 
 where $\mathcal{M}$ is the Brownian meander and $\sigma^2:=\Var(S_1)=\E(\sum_{|x|=1} V^2(x) e^{-V(x)})< + \infty$ by \eqref{hyp0}. So finally we obtain in this case:  for $n$ large enough,
\begin{align}
\e{Z^n}  \geq \frac{1}{C\log n}.
\end{align}

\smallskip
 
\noindent $\bullet $ For $\e{(Z^n)^2}$, we first introduce the sequence of variables
$$
\forall k\in\N,\ M_k:= \sum_{ |x|=|y|={k}} e^{-t_0 V(x)-t_0 V(y)} \1{\underline{V}(x)\wedge \underline{V}(y)\geq -B }\ .
$$ 
For any $k\in\N$,
\begin{align*}
M_{k+1}\leq&\sum_{ |u|=|v|=k} e^{-t_0V(u)-t_0V(v)} \1{\underline{V}(u)\wedge \underline{V}(v)\geq -B}\sum_{\substack{x\text{ s.t. }x^*=u\\y\text{ s.t.} y^*=v}}e^{-t_0V_u(x)-t_0V_v(y)}\\
=&\sum_{\substack{u\neq v\\|u|=|v|=k}} e^{-t_0 V(u)-t_0 V(v)} \1{\underline{V}(u)\wedge \underline{V}(v)\geq -B}\sum_{\substack{x\text{ s.t. }x^*=u\\y\text{ s.t.} y^*=v}}e^{-t_0 V_u(x)-t_0 V_v(y)}\\
&+ \sum_{|u|=k} e^{-2t_0 V(u)} \1{\underline{V}(u)\geq -B}\sum_{\substack{x,y\text{ s.t. }\\x^*=y^*=u}}e^{-t_0 V_u(x)-t_0V_u(y)}\ .
\end{align*}
Then taking conditional expectation with respect to $\mathcal{E}_k := \sigma\paren{\T_k,\ (V(x), |x| \leq k)}$ , the above expression gives
\begin{align*}
\e{M_{k+1}|\mathcal{E}_{k}}&\leq M_{k}\e{\sum_{|x|=1}e^{-t_0 V(x)}}^2+\sum_{|u|=k} e^{-2t_0 V(u)} \1{\underline{V}(u)\geq -B}\e{\sum_{|x|=|y|=1}e^{-t_0 V(x)-t_0 V(y)}}\\ %-1
%&\leq M_{k}+C\sum_{|u|=k} e^{-2t_0 V(u)} \1{\underline{V}(u)\geq -B}
&\leq M_{k}+C\sum_{|u|=k} e^{-t_0 V(u)}\ .
\end{align*}
Note that $\e{\sum_{|x|=|y|=1}e^{-t_0 V(x)-t_0 V(y)}}=\e{(\sum_{|x|=1}e^{-t_0 V(x)})^2}$ is finite thanks to \eqref{hyp1}. And we get recursively, as $M_0=1$,
\begin{align*}
\e{M_{k+1}}\leq 1+C\e{\sum_{|u|\leq k} e^{-t_0 V(u)}}=1+C(k+1). 
\end{align*}
This gives the bound for the second moment: $\e{(Z^n)^2}\leq \e{M_{\ell_n}}\leq C\ell_n$. Collecting this bound together with the one for $\e{Z^n}$ in \eqref{eqconc}, this concludes the proof.
\end{proof}

\section{Lower and Upper bound for the heavy range of recurrent walks \label{secSR}}

In this section, we prove Theorems \ref{thm1} and \ref{Cortpsdet}. We consider separately lower and upper bounds. The upper bounds are easier to obtain than the lower ones so they are treated for all cases at the beginning of this section in Proposition \ref{proprange}. For the lower bounds we treat separately the contribution coming from vertices with low potential in Proposition \ref{proprange0} and the contribution coming from vertices with high potential in Proposition \ref{proprange0bis}. It turns out that for slow random walks, that is to say random walks with logarithmic behavior, the contribution coming from vertices with low potential is sufficient to obtain the asymptotics for the $\log$-heavy range. Conversely, for fast but sub-diffusive cases, that is to say when $1<\kappa < 2$, then only vertices with high potential contribute. Finally for diffusive cases, that is to say for $\kappa\geq2$ then either vertices with low or high potential contribute depending on the value of $\theta$.

Remind that $\Pe_x$ stands for the probability where the environment is fixed and the index $x$ stands for the starting point of the random walk. To obtain the bounds, the following environment-related variable is essential:
\begin{align}
\forall x\in\T,\quad H_x:=\sum_{y\leq x}e^{V(y)-V(x)}. \label{Hx}
\end{align}
Indeed the important quenched probabilities below are related to $H_x$. For any $x \in \T$, let $T_x:= \inf\{k\geq0, \ X_k= x\}$ be the hitting time of vertex $x$. As the walk $\seq{X}$ is recurrent, the expressions of the probability of hitting $x$ before $e^*$ starting from the root and the probability of hitting $x$ before $e^*$ starting from $x^*$ are the same as for a one-dimensional walk: the restriction of $\seq{X}$ to the path $\llbracket e^{*},x\rrbracket$. So a standard result for one-dimensional random walks in random environment (see \cite{Golosov}) yields, for any $x\in\T$,
\begin{align}\label{defa}
a_x:=\Pe _{e}(T_x<T_{e^{*}})=1/ \sum_{z\in\llbracket e,x\rrbracket} e^{V(z)}=\frac{e^{-V(x)}}{H_x}
\end{align}
and
\begin{align}\label{defb}
b_x:=\Pe _{x^*}(T_x<T_{e^{*}})=\sum_{z\in\llbracket e,x^*\rrbracket} e^{V(z)}/ \sum_{z\in\llbracket e,x\rrbracket} e^{V(z)}=1-\frac{1}{H_x}\ .	
\end{align}

\subsection{Upper bounds}

The main results of this section is the following proposition, it gives the upper bounds for $\rang[n^{\theta}]$ depending on the value of $\theta$. The proof is quite short compared to the lower bounds and gives a good idea of what can be expected for the lower bounds.

\begin{prop} \label{proprange} Assume \ref{ass1} and fix $\theta\in[0,1)$. For any $\delta>0$, there is a constant $\epsilon>0$ such that, for $n$ large enough,
\begin{itemize}
\item  in all cases but $\psi(1)=0$ and $\psi'(1) < 0$,
\begin{align*}
\p{\rang[n^{\theta}]> n^{t_0(1-\theta)+\delta}} \leq n^{-\epsilon}\ , 
\end{align*}
\item  if  instead $\psi(1)=0$ and $\psi'(1) < 0$, 
\begin{itemize}
	\item if $ 1<\kappa \leq 2$, $\displaystyle\quad \p{\rang[n^{\theta}]  > n^{\kappa (1-\theta) +\delta}} \leq n^{-\epsilon}\ $,
\item if $\kappa\in(2,\infty]$, $\displaystyle\quad \p{\rang[n^{\theta}]  > n^{(2-\kappa\theta)\vee(1-\theta) +\delta}} \leq n^{-\epsilon}\ $. 
\end{itemize}
 \end{itemize}
\end{prop}

To prove the proposition, we have to control the largest generations visited by the walk before $n$ returns to $e^*$: $ X^*_{\Tn}= \max_{ 0 \leq k \leq \Tn }|X_k|$. 
The different behaviors of this random variable, which depend on $\psi$, have been studied in  \cite{HuShi10a}, \cite{HuShi10} and \cite{HuShi10b}. In the following lemma we present a simpler version of their results adapted to our purpose:
\begin{lemm} \label{asympbeh}
 If  $\inf_{s \in [0,1]} \psi(s)=0$ and $\psi'(1) \geq 0$ there exist $B,b>0$ such that for $n$ large enough,
\begin{align*}
\p{ X^*_{\Tn}  \geq L_n} \leq  n^{-b}\  \textit{with } L_n=B (\log n)^{3}.
\end{align*}
Otherwise if $\psi'(1) < 0$ for any $\epsilon>0$, there exists $b>0$ such that, for $n$ large enough,
\begin{align*}
\p{ X^*_{\Tn}  \geq   L_n}  \leq n^{-b}\ \textit{with } L_n=n^{\min(\kappa-1,1) + \epsilon} .
\end{align*}
Finally  if $\inf_{s \in [0,1]} \psi(s)<0$ there exist $B>0$ and $b>0$ such that for $n$ large enough, 
\begin{align*}
%& \tilde H_{n, \alpha} := \sup\{k>0, \ \inf_{|x|=k} N_n(x) \geq \alpha \}, \\
\p{X^*_{\Tn}  \geq L_n } \leq  n^{-b}\ \textit{with } L_n=B\log n.
\end{align*}
\end{lemm}

\begin{proof} By the strong Markov property, $\pe{X_{\Tn}^*  \geq L_n}=1-\paren{1-\pe{T_{|L_n|}<T_{e^*}}}^n$ where $T_{|L_n|}$ is the hitting time of generation $L_n$:
$$
T_{|L_n|}=\inf\set{t\geq0\ /\ \exists x\in\T, |x|=L_n \text{ and } X_t=x}
$$
where $\inf\emptyset=\infty$.
As for any $x\in[0,1]$, $1-(1-x)^n\leq nx$, integrating the previous equality with respect to the distribution of $\mathcal{E}$ gives
\begin{align*}
	\p{X_{\Tn}^*  > L_n}\leq n\p{T_{|L_n|}<T_{e^*}}.
\end{align*}
This probability has been intensively studied in \cite{HuShi10a}, \cite{HuShi10}, \cite{HuShi10b} and \cite{Hu2017}. More precisely, when $\inf_{s \in [0,1]} \psi(s)=0$: if  $\psi'(1) \geq 0$ then Equation (5.4) in \cite{HuShi10b} gives, for $B$ large enough, the existence of a constant $c_1>1$ such that
\begin{align}\label{eq441}
	 \limsup_{n \rightarrow +\infty} \frac{1}{\log n }\log \p{T_{|L_n|}<T_{e^*}} \leq - c_1,\ 
\end{align}
  If instead $\psi'(1) < 0$ then Proposition 4.2 (ii and iii) in \cite{HuShi10} implies 
\begin{align}\label{eq442}
\limsup_{n \rightarrow +\infty} \frac{1}{\log n } \log \p{T_{|L_n|}<T_{e^*}} \leq -1-\epsilon/2, 
\end{align}
(note that Proposition 4.2 is written for regular trees and the return time in $e^*$ is replaced by the return time in $e$ but this does not change the normalization rate). Finally, if $\inf_{s \in [0,1]} \psi(s)<0$, according to the proof of Theorem 1.1 in \cite{HuShi10},  for $B$ large enough there exists $c_2>1$ such that, 
\begin{align}\label{eq443}
\limsup_{n \rightarrow +\infty} \frac{1}{\log n } \log \p{T_{|L_n|}<T_{e^*}} \leq -c_2.
\end{align}
The different results of the lemma are now direct consequences of \eqref{eq441}, \eqref{eq442} and \eqref{eq443}.
\end{proof}

We are now ready to prove the proposition.

\begin{proof}[Proof of Proposition \ref{proprange}]

We first restrict the sum over the whole tree to the first $L_n$ generations where $L_n$ is the sequence introduced in Lemma \ref{asympbeh} corresponding to the assumptions about $\psi$:
\begin{align}\label{LaProb}
\p{\rang[n^{\theta}]> n^{{t_0 (1-\theta)+ \delta}}}\leq \p{ X^*_{\Tn}  \geq L_n}+n^{-t_0 (1-\theta)- \delta} \E{\sum_{ |x|  \leq L_n}\pe{\NT\geq  n^{\theta}}}
\end{align}
where the second term in the above upper bound comes from Markov inequality. Thanks to  Lemma \ref{asympbeh}, we only have to bound the last expectation $\e{\sum_{ |x|  \leq L_n}\Pe(\NT \geq  n^{\theta})}$. The proof is different whether $L_n$ is a power of $n$ or a power of $\log n$.\\
\textit{ {\bf Case 1:} $L_n$ is of the order $(\log n)^p$, $p>0$.}\\
We split the sum into two terms depending on the value of $V(x)$: define 
$$
\Sigma_1:=\E\Big[\sum_{ |x|  \leq L_n}\Pe(\NT \geq  n^{\theta})\1{V(x) \leq (1-\theta) \log n}\Big]$$
and
$$\Sigma_2:=\E\Big[\sum_{ |x|  \leq L_n}\Pe(\NT \geq  n^{\theta})\1{V(x) > (1-\theta) \log n}\Big].
$$
For $\Sigma_1$, we bound $\Pe(\NT \geq  n^{\theta})$ by 1 and use the many-to-one Formula, Lemma \ref{manytoone}, with $t=t_0$:
\begin{align*}
\Sigma_1&\leq \E\Big[\sum_{ |x|  \leq L_n}\1{V(x) \leq (1-\theta)\log n}\Big]\leq \E\Big[\sum_{ |x|  \leq L_n}e^{t_0((1-\theta)\log n -V(x))}\Big]\\
&=n^{t_0(1-\theta)}\sum_{i=0}^{L_n}\e{e^{t_0(S_i -S_i)}}=n^{t_0(1-\theta)}(L_n+1).
\end{align*}
For $\Sigma_2$, we first compute the expectation of $\NT$ at fixed environment:
\begin{align} \label{Lamoy}
 \ee{\NT}&=\sum_{i=1}^n\ee{\NTi{x}{i}-\NTi{x}{i-1}}=n\ee{\NTi{x}{1}}=n\frac{a_x}{1-b_x}=ne^{-V(x)}
\end{align}
where $a_x$ and $b_x$ have been defined in \eqref{defa} and \eqref{defb}. Now we can apply Markov inequality to $\Pe(\NT \geq  n^{\theta})$ and use many-to-one formula with $t=t_0\leq1$:
\begin{align*}
\Sigma_2 & \leq n^{(1-\theta)} \E\Big[\sum_{|x| \leq L_n} e^{- V(x)} \1{V(x) > (1-\theta)\log n}\Big] \\
 & \leq n^{(1-\theta)} e^{(t_0-1) (1-\theta) \log n} \E\Big[\sum_{|x| \leq L_n} e^{- t_0 V(x)}  \1{V(x) > (1-\theta) \log n} \Big] \leq n^{t_0 (1-\theta) } (L_n+1).
\end{align*}
Inserting the upper bounds of $\Sigma_1$ and $\Sigma_2$ in \eqref{LaProb} concludes the proof in this first case.

\smallskip

\noindent \textit{{\bf Case 2:} $L_n$ is of the order $n^{\min(\kappa-1,1)+\epsilon}$.}\\ 
 Note that in this case, $\psi$ satisfies $\inf_{s \in [0,1]} \psi(s)=0$ with  $\psi'(1) < 0$, so $t_0=1$. We assume in the definition of $L_n$ that  $\epsilon=\delta/2$ and take some real $d>0$ which value will be fixed later. First we split the expectation $\e{ \sum_{ |x|\leq L_n}\1{ \NT \geq n^{\theta}} }$ into  two parts $\Sigma_1$ and $\Sigma_2$, depending on the values of $\red{V}(x)$:
$$
 \Sigma_1:=\E\Big[  \sum_{ |x| \leq L_n} \pe{{ \NT \geq n^{\theta}}}\1{{V}(x) \leq d \log n}\Big]
 $$
  and
  $$\Sigma_2:=\E\Big[  \sum_{ |x| \leq L_n} \pe{{ \NT \geq n^{\theta}}}\1{{V}(x) > d \log n}\Big]\ .
 $$
 We first deal with $\Sigma_1$. Equality \eqref{Lamoy}, Markov inequality  and many-to-one Formula, Lemma \ref{manytoone}, with $t=1$ yield 
\begin{align*}
\Sigma_1& \leq \E\Big[\sum_{|x| \leq L_n}  ne^{-V(x)}n^{-\theta} \1{{V}(x) \leq d \log n}\Big]= n^{1-\theta}   \sum_{j=1}^{L_n} \P\Big( {{S}_j \leq d \log n}\Big).
\end{align*}
 As $\psi'(1)< 0$ we can choose a $r>0$ in such way that $\psi(1+r)<0$. Let $A>0$. When  $L_n$ is larger than $A\log n$,
\begin{align}
\sum_{j=1}^{L_n} \p{ {S}_j \leq d \log n } %& \leq \sum_{j=1}^{A \log n}\p{ S_j \leq d \log n } +\sum_{j=A \log n} ^{L_n} \p{  S_j \leq  d \log n }   \nonumber \\  
 & \leq A \log n+\sum_{j=A \log n} ^{L_n} \p{ e^{-r{S_j}} \geq e^{-r d \log n} }\label{evSj} \\
 & \leq A \log n+\sum_{j=A \log n} ^{L_n} e^{r d \log n} e^{ \psi(1+r)j},\nonumber \end{align}
where we have used the equality $\e{e^{-r S_1}}=e^{ \psi(1+r)}$ obtained from Lemma \ref{manytoone}. So the last sum is smaller than $e^{r d \log n} e^{ A \psi(1+r) \log n} $, and choosing $A$ large enough, this converges to $0$. % we get that $\sum_{j=A \log n} ^{L_n} e^{r d \log n} e^{ \psi(1+r)j}=o(1).$ 
Therefore when $\psi'(1)<0$, the following bound holds for some constant $C>0$:
\begin{align}\label{sig1}
\Sigma_1 \leq C  n^{1-\theta}\log n\ .
\end{align}

We now have to deal with $\Sigma_2$. For $x\in\T$, consider $E^{n}_x :=\sum_{i=1}^n \1{  \exists k\in [T^{i-1},T^{i}) ,\ X_k=x }$ the number of excursions where $x$ has been hit. We split once again $\Sigma_2$ into two other sums depending on the value of $E^{n}_x$:  
$$ 
 \Sigma_{2,1}=\E\Big[\sum_{ |x| \leq L_n} \1{ \NT \geq n^{\theta}}  \1{{V}(x)>d \log n} \1{E^{n}_x  \geq 2}\Big]
 $$
 and
$$\Sigma_{2,2}=\E\Big[\sum_{ |x| \leq L_n} \1{ \NT \geq n^{\theta}}  \1{{V}(x)>d \log n} \1{E^{n}_x  =1}  \Big]\ . 
 $$
 We first prove that, for $d$ large enough, the walk will be able to reach a vertex $x$ satisfying $V(x) > d \log n$ during a unique excursion $[T^{i-1},T^{i})$. Under $\Pe$, $E^{n}_x$ follows the binomial distribution $\mathcal{B}(n,a_x)$, thus
 \begin{align*}
 \pe{E^{n}_x\geq2}\leq\ee{E^{n}_x}-\pe{E^{n}_x=1}&=na_x(1-(1-a_x)^{n-1})\leq n^2a_x^2\leq n^2e^{-2 {V}(x)}
 \end{align*}
 and many-to-one Formula yields 
 \begin{align} \label{sig21}
\Sigma_{2,1}  \leq n^2 \sum_{j=1}^{L_n} \e{e^{S_j-2 {S}_j}\1{ {S}_j >d \log n}}\leq L_n n^{2-d}\leq 1
 \end{align}
 for $d$ large enough. For $\Sigma_{2,2}$, we first notice that 
\begin{align*}
\pe{\NT \geq n^{\theta}, E^{n}_x =1}= \sum_{i=1}^n \Pe(\NTi{x}{i}-\NTi{x}{i-1} \geq n^{\theta},\ \forall j \neq i,\ N^{(j)}_x-N^{(j-1)}_x=0).
\end{align*} 

In particular,  $\Pe(\NT \geq n^{\theta}, E^{n}_x =1) \leq n a_x (b_x)^{n^{\theta}}$ and thanks to many-to-one Formula, we get
\begin{align*} 
\Sigma_{2,2} &\leq n \E\Big[\sum_{ |x| \leq L_n} a_x (b_x)^{n^{\theta}}  \1{{V}(x)>d \log n}\Big]\leq n \sum_{j=1}^{L_n} \e{ \frac{1}{H_j^S} \bigg(1-\frac{1}{H_j^S}\bigg)^{n^{\theta}}}
 \end{align*}
 where $H_j^S={\sum_{m=1}^j e^{S_m-S_j}}$ is the random variable defined in Lemma \ref{compHS}. Now remark that if $H_j^S \leq  n^{\theta} / (3 \log n)  $, then $\Big(1-1/H_j^S  \Big)^{n^{\theta}} \leq 1/n^3$ and so 
$$\sum_{j=1}^{L_n} \e{ \frac{1}{H_j^S} \bigg(1-\frac{1}{H_j^S}\bigg)^{n^{\theta}} \1{H_j^S \leq   n^{\theta}/  (3 \log n) } } \leq L_n/n^3$$  
and according to Lemma \ref{compHS},
\begin{align*} 
 \e{ \frac{1}{H_j^S} \bigg(1-\frac{1}{H_j^S}\bigg)^{n^{\theta}} \1{H_j^S >   n^{\theta}/3 \log n) } }  
 & \leq  \frac{3  \log n}{n^{\theta}} \p{{H^S_{j} >   \frac{n^{\theta}}{3 \log n} }} \leq C\frac{(\log n)^{\kappa}}{n^{\kappa\theta}}.
 \end{align*}
Finally, 
\begin{align}\label{sig22n}
\Sigma_{2,2} &\leq C (\log n)^{\kappa} L_n n^{1-\kappa \theta}.%\leq C n^{(2-\kappa)\wedge 0+ c\kappa+\delta/2}
\end{align} 
The bounds obtained for $\Sigma_1$, $\Sigma_{2,1}$ and $\Sigma_{2,2}$ in \eqref{sig1}, \eqref{sig21} and \eqref{sig22n} give the result in this second case.
\end{proof}
 \subsection{Lower bounds}
 In this section, we prove two propositions: the first one gives a lower bound for the heavy range $\rang[n^{\theta}]$ for any cases but is only optimal for the slowest cases, that is to say for random walks with logarithmic behavior. This first proposition is obtained by  considering vertices with low potential (see definition of set $A_n$ below). The second one, Proposition \ref{proprange0bis}, which is more technical to obtain, deals only with the fast cases, that is to say when $\psi(1)=0$ and $\psi'(1)<0$ and focuses on vertices with high potential.

\begin{prop}  \label{proprange0}
 Assume \ref{ass1} and fix $\theta\in[0,1)$. For any $ 0 <\delta <t_0(1-\theta)$, there exists a constant $\epsilon>0$ such that for $n$ large enough,
\begin{align*}
\quad \P^*\Big(\sum_{|x|\leq(\log n)^3}  \1{ \NT \geq n^\theta} < n^{t_0(1-\theta)-\delta}\Big) \leq n^{-\epsilon}\ .
\end{align*}
\end{prop}
  
Remark that the upper bound for $\rang[n^{\theta}]$ given in Proposition \ref{proprange} for the case $\inf_{s \in [0,1]} \psi(s) = 0$ and $\psi'(1) < 0$ can be larger than the lower bound obtained in the above proposition. This means that we ignore vertices with high potential that are of great importance in this case. They are treated in  Proposition \ref{proprange0bis}.

 \begin{proof} 
  Consider the subset 
\begin{align*}
\forall n\geq1,\quad A_n&:=\set{x\in\T,\ \overline{V}(x)\leq (1-\theta-\delta)\log n,\   \text{ and }\ |x|\leq (\log n)^3}.
\end{align*}
Let us assume for the moment that for $n$ large enough,
\begin{align}\label{eqAn}
\P\Big(\sum_{x\in A_n}\1{\ \NT<n^{\theta}}\geq 1 \Big) \leq e^{ -n^{\theta}}\ ,
\end{align}
 then
 $$\P\Big(\sum_{|x|\leq(\log n)^3}  \1{ \NT \geq n^\theta} < n^{t_0(1-\theta)-\delta}\Big)\leq \P\Big(|A_n|< n^{t_0(1-\theta)-\delta}\Big) +e^{ -n^{\theta}}
 $$
 and the control on the cardinal of $A_n$ given in Proposition \ref{propenv1}, taking $c=1- \theta-\delta$, concludes the proof the proposition.

So let us now prove \eqref{eqAn}. Fix an environment $\mathcal{E}=(\T,\boldsymbol{\omega})$ and consider a vertex $x\in\T$.
 By the strong Markov property, the sequence  $(\NTi{x}{i} -\NTi{x}{i-1} , i \geq 1)$ is, under $\Pe=\Pe_e $, an i.i.d. sequence  with law given by $\NTi{x}{1}$ (remark that $N^{(0)}_x =0$). Applying exponential Markov inequality, we have for any $\eta>0$ and $\alpha>0$, 
\begin{align}
\pe{\NT  \leq \alpha}=\P\Big(\sum_{i=1}^n  [\NTi{x}{i} -\NTi{x}{i-1}] < \alpha\Big)& =\pe{ e^{-\eta \sum_{i=1}^n  [\NTi{x}{i} -\NTi{x}{i-1}]} > e^{- \eta \alpha } } \nonumber \\
&  \leq \exp\paren{\eta \alpha +n\log \ee{e^{-\eta  \NTi{x}{1})}} } \label{upboundP}.
\end{align}
It is easy to see, by the strong Markov property,  that the distribution of $ \NTi{x}{1}$ under $\Pe_x$ is geometrical with parameter $1-b_x$, so that 
\begin{align*}
\log \ee{\exp(-\eta  \NTi{x}{1})}&=\log\paren{(1-a_x)+a_x\Ee_x\cro{\exp(-\eta  \NTi{x}{1})} }\\
&=\log\paren{1-a_x+ a_x(1-b_x)/(e^{\eta}- b_x)}\\
&\leq -a_x+ a_x(1-b_x)/(1+\eta- b_x)\ .
\end{align*}
Then coming back to \eqref{upboundP}  with $\eta=(1-b_x)$, we obtain
\begin{align}\label{TL}
\pe{\NT  \leq \alpha} \leq \exp\paren{-\frac{na_x}{2}+(1-b_x)\alpha}.
\end{align}
The equations of $a_x$ and $b_x$ given in \eqref{defa} and \eqref{defb} imply that $a_x/(1- b_x)=e^{-V(x)}$. Together with \eqref{TL}, this gives 
\[ \pe{\NT  \leq \alpha}  \leq e^{-\frac{a_x}{2}\paren{n-e^{V(x)}\alpha}}\ . \] 
 As for any $x \in A_n$, $e^{-V(x)}\geq n^{-(1-\theta-\delta)}$ and $a_x\geq 1/(n^{1-\theta-\delta}(\log n)^3)$, the above inequality implies that
\begin{align*}
%\pe{\exists x\in A_n,\ \NT<n^{\theta} } &= 
\Pe\Big(\sum_{x\in A_n}\1{\ \NT<n^{\theta}}\geq 1\Big) &   \leq \sum_{x\in A_n} \pe{\NT  \leq n^{1-\delta}e^{-V(x)}}   \leq  |A_n| e^{ -{n^{\theta+\delta}(1-n^{-\delta})/(\log n)^3}}\ .
\end{align*}
 Integrating with respect to the distribution of $\mathcal{E}$, we obtain for $n$ large enough:
$$
\P\Big(\sum_{x\in A_n}\1{\ \NT<n^{\theta}}\geq 1 \Big) \leq C\e{|A_n|}e^{-n^{\theta+\delta/2}}\ .
$$ 
 Then we only need an appropriate upper bound for $\e{|A_n|}$ to complete the proof. This is a direct consequence of many-to-one Formula. Indeed,  taking $t=t_0$ in Lemma \ref{manytoone}, we get 
\begin{align*}
\e{|A_n|}=&\E \Big [ \sum_{|x| \leq (\log n)^3} \1{\overline{V}(x) \leq (1-\theta-\delta)\log n }  \Big ] \\
 =&\sum_{k=1}^{\lfloor(\log n)^3\rfloor}\E\Big [ e^{t_0 S_{k} +k \psi(t_0)} \1{\overline{S}_{k} \leq  (1-\theta-\delta)\log n} \Big ]\leq  (\log n)^3 n^{t_0 (1-\theta-\delta)}\ .
\end{align*}
\end{proof}

The following proposition deals with the rapid cases $(\psi(1)=0,\psi'(1)<0)$: we treat the vertices with large potential left aside in the previous proposition. Quite technical, the proof is decomposed in essentially four Lemmata.

 \begin{prop}  \label{proprange0bis}
 Assume \ref{ass1} and suppose we are in the case 
 $$\inf_{s \in [0,1]} \psi(s)=\psi(1)=0\ ,\ \psi'(1)<0\text{ and }\kappa=\inf\set{s>1,\ \psi(s)=0}<\infty\ .$$
 Define $\zeta:=(\kappa-1)^{-1}\wedge 1$ and  consider a real $\theta\in[0,\zeta)$. For any constant $ 0 <\delta <(\kappa-1)(\zeta-\theta)$, there exists a positive number $\epsilon$ such that for $n$ large enough,
\begin{align*}
&\p[*]{\rang[n^{\theta}]  \leq n^{2\wedge\kappa -\kappa\theta-\delta}} \leq n^{-\epsilon}. 
\end{align*}
\end{prop}
 
 In the sequel it is enough to consider  generations slightly smaller than typical visited generations for this  cases (that is $n^{(\kappa-1)\wedge 1}$, as recalled in Lemma \ref{asympbeh}), so we define
$$
\Ldel:=n^{(\kappa-1)\wedge1-\delta/2}=n^{(\kappa-1)\zeta-\delta/2}\ .  
$$ 
  
First let us introduce the set $\Gamma_n=\Gamma_n^1\cap\Gamma_n^2$ where
\begin{align*}
\Gamma_n^1&=\set{x\in\T\ /\ n^{\theta}\leq H_x\leq n^{\theta+\delta/16},\ V(x)\geq 4\log n}\text{ and}\\
\Gamma_n^2&=\Big\{x\in\T\ /\ \sum_{z\leq x} H_z\leq n^{1-7\delta/16}\text{ and } \forall z\leq x,\ H_z\leq n^{\zeta}\Big\}	,
\end{align*}
with $H_x$ given in \eqref{Hx}.
The set $\Gamma_n$ contains the vertices which are the main contributors to the heavy range. Remark that, when $\kappa\leq 2$, the condition $\set{\forall z\leq x,\ H_z\leq n^{\zeta}}$ is implied by the condition on the sum $\sum_{z\leq x} H_z$ and is therefore useless in this case.
 
 The set $\Gamma_n$ is such that the walk visits most of these vertices more than $n^{\theta}$ times with a large probability.
Then, as
$$
\rang[n^{\theta}]\geq \sum_{|x|\leq \Ldel} \1{\NT\geq n^\theta}\1{x\in\Gamma_n},
$$ 
we only have to obtain a lower bound for the above sum.
   We first prove in the following lemma that, with a large probability, the walk reaches a given vertex of $\Gamma_n$ during a single excursion $[T^{i-1},T^{i})$.
\begin{lemm}\label{eqindep}
For any $x\in\T$, recall $E^{n}_x :=\sum_{i=1}^n \1{  \exists k\in [T^{i-1},T^{i}) ,\ X_k=x}$ the number of excursions where the walk hits vertex $x$. Then, for $n$ large enough,
	$$
	\P\Big(\sum_{|x|\leq \Ldel}\1{E^{n}_x\geq2}\1{x\in\Gamma_n}\geq 1\Big)\leq n^{-1}\ .
	$$
\end{lemm}
\begin{proof}
Under $\Pe$, $E^{n}_x$ follows the binomial distribution $\mathcal{B}(n,a_x)=\mathcal{B}(n,e^{-V(x)}/H_x)$ so 
 \begin{align*}
 \pe{E^{n}_x\geq2}\leq\ee{E^{n}_x}-\pe{E^{n}_x=1}&=na_x(1-(1-a_x)^{n-1})\leq n^2a_x^2\ .
 \end{align*}
Moreover, for $x\in\Gamma_n$, $n^2a_x^2=n^2e^{-2V(x)}H_x^{-2}\leq n^{-2}e^{-V(x)}$ and
\begin{align*}
	\E\Big[\sum_{|x|\leq \Ldel}\pe{E^{n}_x\geq2}\1{x\in\Gamma_n}\Big] \leq n^{-2} \E\Big[\sum_{ |x| \leq \Ldel} e^{-V(x)}\Big]\leq n^{-2} \Ldel\leq n^{-1}\ .
\end{align*}
And Markov inequality gives the result of the lemma.	
\end{proof}
 
Previous lemma shows that there is independence between the contributions of the different  excursions. This is an important fact to obtain  the lower bound in Lemma \ref{lemtcheb}. We also need  a control on the mean $\ee{\NTi{x}{1}\NTi{y}{1}}$ given in Lemma \ref{ConVar} below; this result is similar to Lemma 5.2 in \cite{HuShi15}, but we give here a short proof for completeness.
\begin{lemm}\label{ConVar}
For any $x, y\in\T$,
\begin{align}
%\ee{\paren{\NTi{x}{1}}^2}&=e^{-V(x)}\paren{2H_x-1}\label{mom21}\\
\text{if $x\leq y$, }\quad \ee{\NTi{x}{1}\NTi{y}{1}}&=e^{-V(y)}\paren{2H_x-1}\label{mom22}\\
\text{ else }\quad \ee{\NTi{x}{1}\NTi{y}{1}}&=2e^{-V(x)-V(y)+V(x\wedge y)}H_{x\wedge y}.\label{mom23}
\end{align}
\end{lemm}
\begin{proof} A direct calculation leads to \eqref{mom22} when $x=y$. We proceed by induction on the generations for the general case. The result is obvious for $x=y=e$. Suppose now that Equations \eqref{mom22} and \eqref{mom23} are true for any $|x|\vee|y|\leq m$  for some integer $m$ and consider two vertices $x$ and $y$ such that $|x|\vee|y|\leq m+1$. If $|x|\vee|y|\leq m$, the result is direct. If $|x|\leq m$ and $|y|=m+1$, we consider the $\sigma$-algebra $\mathcal{G}_m=\sigma\paren{\mathcal{E},(\NTi{x}{1},|x|\leq m)}$. Given $\mathcal{G}_m$, $\NTi{y}{1}$ follows the negative binomial distribution $BN(\NTi{y^*}{1},1/(1+e^{-(V(y)-V(y^*))}))$. Then,
\begin{align*}
\ee{\NTi{x}{1}\NTi{y}{1}\big|\F_m}=\NTi{x}{1}\NTi{y^*}{1}e^{-(V(y)-V(y^*))}.
\end{align*}
 Taking the expectation leads to \eqref{mom22}. Finally, if $|x|=|y|=m+1$, same kinds of arguments show that
 $$
 \ee{\NTi{x}{1}\NTi{y}{1}\big|\F_m}=\NTi{x^*}{1}\NTi{y^*}{1}e^{-(V(x)-V(x^*))-(V(y)-V(y^*))},
 $$
  if $x^*\neq y^*$ and otherwise
 $$
 \ee{\NTi{x}{1}\NTi{y}{1}\big|\F_m}=\NTi{x^*}{1}(\NTi{x^*}{1}+1)e^{-(V(x)-V(x^*))-(V(y)-V(y^*))}
 $$
 and the result follows.
 \end{proof}

\begin{lemm}\label{lemtcheb} 
There is a constant $C$ such that for $n$ large enough,
\begin{align*}
&\pe{\sum_{i=1}^n\sum_{|x|\leq \Ldel}\1{\NTi{x}{i}-\NTi{x}{i-1}\geq n^{\theta}}\1{x\in\Gamma_n}\leq n^{1-\theta-\delta/8}\sum_{|x|\leq \Ldel}e^{-V(x)}\1{x\in\Gamma_n}}\\
\leq &\frac{C}{n^{1-\delta/8}} \frac{\sum_{|x|,|y|\leq \Ldel}H_{x\wedge y}e^{-V(x)-V(y)+V(x\wedge y)}\1{x,y\in\Gamma_n}}{\paren{\sum_{|x|\leq \Ldel}e^{-V(x)}\1{x\in\Gamma_n}}^2}, \quad \P\text{-a.s.}
\end{align*}
\end{lemm}

\begin{proof}
As the excursions are i.i.d. under $\Pe$, we first use Tchebytchev inequality with probability measure $\Pe$ to obtain:
\begin{align}
&\pe{\sum_{i=1}^n\sum_{|x|\leq \Ldel}\1{\NTi{x}{i}-\NTi{x}{i-1}\geq n^{\theta}}\1{x\in\Gamma_n}\leq \frac{n}{2}\sum_{|x|\leq \Ldel}\pe{\NTi{x}{1}\geq n^{\theta}}\1{x\in\Gamma_n}} \nonumber \\
\leq& \frac{4}{n}\frac{\sum_{|x|,|y|\leq \Ldel}\pe{\NTi{x}{1}\wedge \NTi{y}{1}\geq n^{\theta}}\1{x,y\in\Gamma_n}}{\paren{\sum_{|x|\leq \Ldel}\pe{\NTi{x}{1}\geq n^{\theta}}\1{x\in\Gamma_n}}^2}\ , \label{BigPro}
\end{align}
where the numerator comes from the computation of the variance at fixed environment: 
\begin{align*} \Var^{\mathcal{E}}\Big( \sum_{i=1}^n\sum_{|x|\leq \Ldel}\1{\NTi{x}{i}-\NTi{x}{i-1}\geq n^{\theta}}\1{x\in\Gamma_n}\Big)& =n \Var^{\mathcal{E}}\Big(\sum_{|x|\leq \Ldel}\1{\NTi{x}{1} \geq n^{\theta}}\1{x\in\Gamma_n} \Big) \\
& \leq n\sum_{|x|,|y|\leq \Ldel}\pe{\NTi{x}{1}\wedge \NTi{y}{1}\geq n^{\theta}}\1{x,y\in\Gamma_n} .  \end{align*}
Remark now that for any $x\in\T$, $\Pe(\NTi{x}{1} \geq n^{\theta})=a_x( b_x)^{\lceil n^{\theta}\rceil} $, with $a_x$ and $b_x$ given in respectively \eqref{defa} and \eqref{defb}
(and $\lceil n^{\theta}\rceil$ stands for the smallest integer larger than $n^{\theta}$). As for $x\in\Gamma_n$, $H_x\in[n^{\theta},n^{\theta+\delta/16}]$, there is a constant $C$ such that for any $n\geq1$,
 \begin{align}\label{eqminmean}
\sum_{|x|\leq \Ldel}\Pe(\NTi{x}{1} \geq n^{\theta})\1{x\in\Gamma_n}\geq \frac{C}{n^{\theta+\delta/16}}\sum_{|x|\leq \Ldel}e^{-V(x)}\1{x\in\Gamma_n}.
 \end{align}
To bound the numerator in \eqref{BigPro}, we use Markov-type inequality:
 \begin{align*}
\sum_{|x|,|y|\leq \Ldel} \pe{\NTi{x}{1}\wedge \NTi{y}{1}\geq n^{\theta}} \1{x,y\in\Gamma_n}\leq&\frac{1}{n^{2\theta}}\sum_{x,y\leq \Ldel}\ee{\NTi{x}{1}\NTi{y}{1}}\1{x,y\in\Gamma_n}\\
\leq& \frac{2}{n^{2\theta}}\sum_{|x|,|y|\leq \Ldel}e^{-V(x)-V(y)+V(x\wedge y)}H_{x\wedge y}\1{x,y\in\Gamma_n}.
\end{align*}
The last bound of the previous equation is obtained by the control of $\ee{\NTi{x}{1}\NTi{y}{1}}$ given in Lemma \ref{ConVar}.
\end{proof}

 The next step is to control the numerator and the denominator of the bound obtained in the previous Lemma. We start with the easiest part, the upper bound for the mean of the numerator.

 \begin{lemm}\label{lemGn} For $n$ large enough,
$$
G_n:=\E\Big[\sum_{|x|,|y|\leq \Ldel}H_{x\wedge y}e^{-V(x)-V(y)+V(x\wedge y)}\1{x,y\in\Gamma_n}\Big]\leq Cn^{1+2(\kappa-1)(\zeta-\theta)-23\delta/16}. 
$$
%with $\zeta=(\kappa-1)^{-1}\wedge 1$.
\end{lemm}

\begin{proof}
Remark first that $G_n$ can be written in the following way
\begin{align*}
G_n&=  \sum_{i=1}^{\Ldel}\E\Big[ \sum_{|u|=i} e^{-V(u)}H_u \Big(\1{u\in\Gamma_n}+\sum_{\substack{v\neq\tilde{v}\\ v^*=\tilde{v}^*=u}}e^{-V_u(v)-V_u(\tilde{v})}\sum_{\substack{x\geq v\\|x|\leq\Ldel}}\sum_{\substack{y\geq \tilde{v}\\|y|\leq\Ldel}} e^{-V_v(x)} e^{-V_{\tilde{v}}(y)}  \1{x,y\in\Gamma_n}\Big)\Big]%\\
\end{align*}
where we recall that for any $z \geq w$, $V_w(z)=V(z)-V(w)$ is the potential centered at $w$. Furthermore, for any $x\in\Gamma_n$ and any ancestor $v \leq x$, the variable $H_x$ can be decomposed as follows
$$H_x=e^{-V_v(x)} H_v + \sum_{v<z\leq x }  e^{V_v(z)-V_v(x)} \leq e^{-V_v(x)} n^{\zeta}+ H^{(v)}_x$$
where $H^{(v)}_x=\sum_{v\leq z\leq x }  e^{V_v(z)-V_v(x)}$.  Now remark that, for $u<v\leq x$ and $u< \tilde{v}\leq y$,
\begin{align*}
\set{u\in\Gamma_n}\subset&B_u\ \text{ and }\ 
 \set{x,y\in\Gamma_n}\subset A_{x} \cap A_{y} \cap B_u\ \text{ where}
\end{align*}
  $A_x:=\{n^{\theta}\leq e^{-V_v(x)} n^{\zeta}+ H^{(v)}_x  \},\ A_y:=\{n^{\theta}\leq e^{-V_{\tilde v}(y)} n^{\zeta}+ H^{({\tilde v})}_y\}$
 and $B_u:=\Big\{\sum_{z\leq u} H_z\leq n^{1-7\delta/16}\Big\}$. So 
\begin{align*}
&  \sum_{|u|=i} e^{-V(u)}H_u \Big(\1{u\in\Gamma_n}+\sum_{\substack{v\neq\tilde{v}\\ v^*=\tilde{v}^*=u}}e^{-V_u(v)-V_u(\tilde{v})}\sum_{\substack{x\geq v\\|x|\leq\Ldel}}\sum_{\substack{y\geq \tilde{v}\\|y|\leq\Ldel}} e^{-V_v(x)} e^{-V_{\tilde{v}}(y)}  \1{x,y\in\Gamma_n}\Big) \\
& \leq  \sum_{|u|=i} e^{-V(u)}H_u \un_{B_u} \Big(1+\sum_{\substack{v\neq\tilde{v}\\ v^*=\tilde{v}^*=u}}e^{-V_u(v)-V_u(\tilde{v})}\sum_{\substack{x\geq v\\|x|\leq\Ldel}} e^{-V_v(x)} \un_{A_{x}}  \sum_{\substack{y\geq \tilde{v}\\|y|\leq\Ldel}}  e^{-V_{\tilde{v}}(y)}  \un_{A_{y}}\Big).%\\
\end{align*}
 The variables $H^{(v)}_x$ and $V_v(x)$ depend only on the  potential centered at $v$ and are independent of the other branches and of the potential of the vertices before $v$. The same is true for $H^{({\tilde v})}_y$ and $V_{\tilde v}(y)$. Therefore, conditioning to variables up to generation $i$ and using many-to-one Formula of Lemma \ref{manytoone} with $t=1$, we obtain
\begin{align}
G_n&\leq  \sum_{i=1}^{\Ldel} \e{H^S_i\1{\sum_{j\leq i}H^S_j\leq n^{1-7\delta/16}}}\Bigg(1+\E\Big[\sum_{\substack{v\neq\tilde{v}\\ v^*=\tilde{v}^*=e}}e^{-V(v)-V(\tilde{v})}\Big]\Big(\sum_{k=0}^{\Ldel-i-1}\phi_{k,\zeta}\Big)^2\Bigg) \label{31}
\end{align}
where  for any integer $k>0$ %, $H^S_k:= \sum_{j\leq k}e^{S_j-S_k}$
and any $b>\theta$, $\phi_{k,b}:=\p{e^{- S_k} n^{b}+H^{S}_{k}>n^{\theta}}$. The upper bound of $\phi_{k,b}$ will be used several times in the sequel so we start by giving a general estimation.  First,  $\phi_{k,b} \leq  \p{e^{- S_k} n^{b}>n^{\theta}/2}+\p{H^{S}_{k}>n^{\theta}/2}$. According to Lemma \ref{compHS}, there is a constant $C>0$ such that for any $n$ large enough, $\p{H^{S}_{k}>n^{\theta}/2}\leq {C}{n^{-(\kappa-1)\theta}}$. And, by Markov inequality, for any $0<\delta<(\kappa-1)(b-\theta)$,
 \begin{align*}
 	\p{e^{- S_k} >\frac{1}{2n^{b-\theta}}}\leq Cn^{(\kappa-1)(b-\theta)-\delta/2}\e{e^{-(\kappa-1-\frac{\delta}{2(b-\theta)})S_k}}=Ce^{k\psi(\kappa-\frac{\delta}{2(b-\theta)})}n^{(\kappa-1)(b-\theta)-\delta/2}.
 \end{align*} 
 Therefore, as $\psi\big(\kappa-\frac{\delta}{2(b-\theta)}\big)<0$, there exists $d>0$ such that 
 \begin{align}
 \phi_{k,b} \leq  {C}({n^{-(\kappa-1)\theta}} +  e^{- d k} n^{ (\kappa-1)(b-\theta)-\delta/2} ).  \label{phikb}
 \end{align}

 \noindent Then applying this inequality in our case that is to say when $b= \zeta$, this gives, as $\Ldel=n^{(\kappa-1)\zeta-\delta/2}\ $,  for $n$ large enough,
  \begin{align*}%\label{majphi}
 \forall i< L_n^\delta,\quad\sum_{k=0}^{\Ldel-i-1}\phi_{k, \zeta}\leq Cn^{(\kappa-1)(\zeta-\theta)-\delta/2}.
 \end{align*}
 Moreover, by Assumption \eqref{hyp1}, $\e{\sum_{v\neq\tilde{v}, |v^*|=|\tilde{v}^*|=1}e^{-V(v)-V(\tilde{v})}}$ is finite. Hence, to conclude the proof of the lemma, we only have to remark that $\sum_{i=1}^{\Ldel}H^S_i\1{\sum_{j\leq i}H^S_j\leq n^{1-7\delta/16}}\leq n^{1-7\delta/16}$. Indeed we add terms to the sum only while the sum stays smaller than $n^{1-7\delta/16}$ so the final result has to be smaller too.% and the bound of the lemma is obtained.
\end{proof}

In the following lemma, we focus on the lower bound for  the denominator in Lemma \ref{lemtcheb}. %For that we prove Lemma \ref{binfmoy} below. 
Note that we also need a technical estimates which are gathered in Lemma \ref{techp}. Its proof is postponed at the end of this section.

 \begin{lemm}\label{binfmoy} There is a constant $\epsilon>0$, such that for $n$ large enough,
$$
\P^*\Big(\sum_{|x|\leq \Ldel}e^{-V(x)}\1{x\in\Gamma_n}< n^{(\kappa-1)(\zeta-\theta)-5\delta/8}\Big)\leq n^{-\epsilon}\ .
$$
\end{lemm}

\begin{proof} 
We use the same strategy as in the proof of Proposition \ref{propenv1}: we cut the tree at an early generation to obtain independence and then use concentration Lemma \ref{lemconcen}. We first use Lemma \ref{ctrlptgen} to control $\overline{V}(u)$ on the very first generations of the tree: there exist $b,\epsilon>0$ such that $\p[*]{\overline {\mathcal{A}}_n }\leq n^{-b}$ where $\mathcal{A}_n=\set{\forall u\in\T_{\epsilon_n},|V(u)| \leq \frac{\delta}{8}\log n}$ and $\epsilon_n=\lfloor \epsilon\log n\rfloor$.
So for $n$ large enough, on the event $\mathcal{A}_n$,  
\begin{align} \forall u\in\T_{\epsilon_n},\ H_u\leq n^{2\delta/8} \epsilon_n \leq  n^{3\delta/8}\ \text{ and }\ \forall u\leq x,\ H_x^{(u)}\leq H_x  \leq  H_x^{(u)}+e^{-V_u(x)} n^{3\delta/8}. \label{set}
\end{align}
Recall that for any  $u\leq x$, $V_u(x)=V(x)-V(u)$ and $H_x^{(u)}=\sum_{u\leq y\leq x}e^{V_u(y)-V_u(x)}$. Fix now some constant $B>0$ and consider the collection of random variables : 
$$\forall n\geq 1,\ \forall u\in\T,\quad Z_u^n=\sum_{\substack{ x>u \\|x|\leq \Ldel}} e^{-V_u(x)} \1{x\in\Gamma_{n,u}}, \text{ with } \Gamma_{n,u}  =\Gamma_{n,u}^1\cap\Gamma_{n,u}^2$$
where the  $\Gamma_{n,u}^i$ are the following new sets of constraints on the environment
\begin{align*}
\Gamma_{n,u}^1&=\set{x\in\T\ /\ n^{\theta}\leq H^{u}_x\leq n^{\theta+\delta/16}/2,\ V_u(x)\geq 2\log n }\text{ and}\\
\Gamma_{n,u}^2&=\Big\{ x\in\T\ /\ \sum_{z\leq x} H^u_z\leq n^{1-3\delta/8}/2\text{ and }\forall u\leq z\leq x,\ V_u(z) \geq -B\text{ and }H^u_z\leq n^{\zeta}/2 \Big \} .	
\end{align*}
The definitions of $\Gamma_n$, $\Gamma_{n,u}$ and inequalities in $\eqref{set}$ imply that on $ \mathcal{A}_n$,
\begin{align*}
 & \sum_{|x|\leq \Ldel}e^{-V(x)}\1{x\in\Gamma_n}  \geq n^{-\delta/8}  \sum_{|u|=\epsilon_n}  Z_u^n\ ,
\end{align*}
so  concentration Lemma \ref{lemconcen} shows that, for some $\epsilon>0$ and $n$ large enough
\begin{align}\label{BorneZ}
\P^*\Big( \sum_{|x|\leq \Ldel}e^{-V(x)}\1{x\in\Gamma_n}\leq n^{-\delta/8}\e{Z^n}  \Big) 
\leq 8 \frac{ \e{\paren{Z^n}^2}}{ \e{Z^n}^2}n^{-\epsilon}+\P^*{(\overline {\mathcal{A}}_n)},
\end{align}
where $Z^n=\sum_{|x|\leq \Ldel-\epsilon_n} e^{-V(x)} \1{x\in \Gamma_{n,e}}.$
The next step is to control $\e{Z^n}$ and $\e{\paren{Z^n}^2}$. 
\smallskip

\noindent $\bullet$ {\it A lower bound for  the  mean  $\e{Z^n}$.} Lemma \ref{manytoone} with $t=1$, gives :
\begin{align*}
&\e{Z^n}\\
=&\sum_{i=0}^{\Ldel-\epsilon_n}\P\Big(n^{\theta}\leq H^S_i\leq n^{\theta+\delta/8}/2,S_i\geq 2\log n,\ \sum_{j\leq i} H^S_j\leq n^{1-3\delta/8}/2, \ \overline{H}^S_i\leq n^{\zeta}/2, \underline{S}_i\geq -B \Big)\\
\geq& \sum_{i=\lceil d \log n\rceil}^{\Ldel-\epsilon_n}\Big(\p{n^{\theta}\leq H^S_i}-\sum_{m=1}^5 p_{m,i}\Big),
\end{align*}
 where $\overline{r}_n:=\max_{j \leq n} r_j$ and $\underline{r}_n:=\min_{j \leq n} r_j $ for any sequence $\seq{r}$ and 
 $d$ is a constant which can be chosen as large as needed. The probabilities  
\begin{align*}
& p_{1,i}:=\p{S_i< 2\log n},\ p_{2,i}:=\p{n^{\theta+\delta/16}/2\leq H_i^S},\ p_{3,i}:=\p{n^{\theta}\leq H^S_i ,  \underline S_i<-B}\\
&	p_{4,i}:=\p{n^{\theta}\leq H^S_i \leq n^{\theta+ \delta}, \overline{H}^S_i> n^{\zeta}/2} \text{ and }\ p_{5,i}:= \P \Big ( n^{\theta}\leq H^S_i,\ \sum_{j\leq i} H^S_j\geq n^{1-3\delta/8}/2 \Big)
\end{align*}
are estimated in the following technical lemma which proof can be found at the end of the section.  
\begin{lemm}\label{techp} %Fix some $B,d>0$  and consider, for any integers $n\geq0$ and $d\log n< i\leq n$  the  probabilities:
%	The following bounds hold: 
If we choose $d$ large enough, there is a constant $C>0$ such that for $n$ large enough and any $i\in\set{d\log n,\dots, n}$,
	\begin{align*}
		p_{1,i}\leq n^{-2(\kappa-1)\theta}&,\quad p_{2,i}\leq Cn^{-(\kappa-1)(\theta+\delta/16)},\quad p_{3,i}\leq C e^{- ((\kappa-1)\zeta-\delta) B} n^{-(\kappa-1) \theta}, \\
		p_{4,i}&\leq Cn^{-(\kappa-1)\theta-\delta/16)}\ \text{ and }\ p_{5,i}\leq Cn^{-(\kappa-1)\theta-\delta/16)}.
	\end{align*}
\end{lemm}

\noindent Collecting the different upper bounds given in the previous lemma, we obtain for large $n$, 
$$\sum_{m=1}^5p_{m,i} \leq  Cn^{- (\kappa-1) \theta}\paren{n^{-\delta/16}+e^{- \lambda B}}.$$ Therefore for $B$ and $n$ large enough,
\begin{align*}
&\e{Z^n} \geq \sum_{i=\left\lceil d \log n\right\rceil}^{\Ldel-\epsilon_n}\Big(\p{n^{\theta}\leq H^S_i}-\sum_{m=1}^5 p_{m,i}\Big) \geq C \Ldel n^{-(\kappa-1)\theta}=C  n^{(\kappa-1)(\zeta-\theta)-\delta/2}.
\end{align*}

\smallskip

\noindent $\bullet$ {\it An upper bound for $\e{(Z^n)^2}$.} We barely use the same arguments as in the proof of Lemma \ref{lemGn}: first  $\e{(Z^n)^2}$ can be written as
\begin{align*}
\sum_{i=1}^{\Ldel} \E\Big[\sum_{|u|=i} e^{-2V(u)}\Big(\1{u\in\Gamma_n}+\sum_{\substack{v\neq\tilde{v}\\ v^*=\tilde{v}^*=u}}e^{-V_u(v)-V_u(\tilde{v})}\sum_{x\geq v}\sum_{y\geq \tilde{v}} e^{-V_v(x)} e^{-V_{\tilde{v}}(y)}  \1{x,y\in\Gamma_n}\Big)\Big].
\end{align*}
Then in the same way as we have obtained \eqref{31} we get
\begin{align*}
\e{(Z^n)^2}&\leq  \sum_{i=1}^{\Ldel} \e{e^{-S_i}\1{S_i\geq -B}}\Bigg(1+\E\Big[\sum_{\substack{v\neq\tilde{v}\\ v^*=\tilde{v}^*=e}}e^{-V(v)-V(\tilde{v})}\Big]\Big(\sum_{k=0}^{\Ldel-i-1}\phi_{k, \zeta}\Big)^2\Bigg)
\end{align*}
and same arguments as in the proof of Proposition \ref{lemGn} (see Equation \eqref{phikb} and below) show that
$$\e{(Z^n)^2} \leq Cn^{2(\kappa-1)(\zeta-\theta)-\delta}\sum_{i=1}^{\Ldel} \e{e^{-S_i}\1{S_i\geq -B}}.$$
All that is left to do is to control the first part of the sum, again by a Markov inequality: as $\psi(2\wedge\kappa-\delta)<0$,
\begin{align*}
	\sum_{i=1}^{\Ldel} \e{e^{-S_i}\1{S_i\geq -B}}\leq e^{B} \sum_{i=1}^{\Ldel} \e{e^{-(1\wedge(\kappa-1)-\delta) S_i}} \leq e^{B} \sum_{i=1}^{\infty}e^{i\psi(2\wedge\kappa-\delta))}<\infty
\end{align*}
 and finally $\e{(Z^n)^2}\leq Cn^{2(\kappa-1)(\zeta-\theta)-\delta}$. The bounds obtained for  $\e{Z^n}$, $\e{(Z^n)^2}$ and Inequality \eqref{BorneZ} conclude the proof.
\end{proof}

We are now ready to prove the main result of this section, Proposition \ref{proprange0bis}.

\begin{proof}[Proof of Proposition \ref{proprange0bis}] 
We only have to collect all the previous results. Thanks to Lemma \ref{eqindep} we can separate the excursions 
and, as $\kappa\wedge2-\kappa\theta=1-\theta+(\kappa-1)(\zeta-\theta)$, thanks to Lemma \ref{binfmoy} we can introduce the variable $\widetilde{\Sigma}_n:=\sum_{|x|\leq \Ldel}e^{-V(x)}\1{x\in\Gamma_n}$. So these two Lemmata imply  
\begin{align*}
& \p[*]{\rang[n^{\theta}]  
 \leq n^{\kappa\wedge 2 -\kappa\theta-\delta}} \leq n^{-1}+n^{-\epsilon} \\
&+ \P^*\Big(\sum_{i=1}^n\sum_{|x|\leq \Ldel}\1{\NTi{x}{i}-\NTi{x}{i-1}\geq n^{\theta}}\1{x\in\Gamma_n}\leq n^{1-\theta-3\delta/8}\widetilde{\Sigma}_n\ ,\ \widetilde{\Sigma}_n \geq n^{(\kappa-1)(\zeta-\theta)-5\delta/8} \Big)
\end{align*}
 for some constant $\epsilon>0$.
Using now the quenched concentration Lemma \ref{lemtcheb}, {as the non-extinction probability is positive,} the probability of the right-hand side of the previous equation can be bounded by
\begin{align*}
&\frac{C}{n^{1-\delta/8}}\e{\paren{\widetilde{\Sigma}_n}^{-2}\sum_{|x|,|y|\leq \Ldel}H_{x\wedge y}e^{-V(x)-V(y)+V(x\wedge y)\1{x,y\in\Gamma_n}}\1{\widetilde{\Sigma}_n \geq n^{(\kappa-1)(\zeta-\theta)-5\delta/8}}}\\
\leq&Cn^{-1-2(\kappa-1)(\zeta-\theta)-11\delta/8}G_n\ .
\end{align*}
Finally, the bound of $G_n$ given in Lemma \ref{lemGn}  shows that 
$$
\p[*]{\rang[n^{\theta}]  \leq n^{\kappa\wedge 2 -\kappa\theta-\delta}}\leq \frac{1}{n}+\frac{1}{n^\epsilon}+\frac{1}{n^{\delta/16}}\ .
$$
This concludes the proof of the proposition.
\end{proof}

We finish the section by the proof of  Lemma \ref{techp} which is essentially based on simple estimates for sums of i.i.d. random variables.

\begin{proof}[Proof of Lemma \ref{techp}] 

The probability  $p_{1,i}$ has already been bounded in this paper, see \eqref{evSj}: as $i \geq \left\lceil d \log n\right\rceil$,  $p_{1,i} \leq n^{-D}$ with $D$ as large as needed if $d$ is large enough. The bound for $p_{2,i}$ is given in Lemma \ref{compHS}.

For the other $p_{.,i}$ we need some more work.  At the end, the calculus for $p_{3,i}$, $p_{4,i}$ and $p_{5,i}$ are very similar so we give the details for $p_{5,i}$ and then sketch the proofs for $p_{3,i}$ and $p_{4,i}$.

Fix some $A>0$ and let $i_n=\max(0,\lceil i-A\log n\rceil)$. Then $H^S_i=H^S_{i_n}e^{-S_{i_n,i}}+H^S_{i_n,i}$ where for $k\leq \ell$,
$$
S_{k,\ell}=S_\ell-S_k\text{ and } H^S_{k,\ell}:=\sum_{j=k+1}^\ell e^{S_{k,j}-S_{k,\ell}}\ .
$$
Note that $H^S_{i_n}$ is independent of $(S_{i_n,i},H^S_{i_n,i})$. First, we have that $p_{5,i}$ is smaller than
\begin{align*}
	  \P\Big(\widetilde{H}_{n,i}\geq n^{\theta} ,\ \sum_{j\leq i_n} H^S_j\geq n^{1-3\delta/8}/4,H^S_{i_n}\leq n^2 \Big) 
	 +\p{H^S_{i_n}\geq n^2}+\P\Big(\sum_{i_n< j\leq i} H^S_j\geq n^{1-3\delta/8}/4 \Big)
\end{align*}
with $\widetilde{H}_{n,i}:=n^2e^{-S_{i_n,i}}+H^S_{i_n,i} $. By independence, the first term above is smaller than   $$\p{\widetilde{H}_{n,i} \geq n^{\theta}} \P\Big(\sum_{j\leq i_n} H^S_j\geq n^{1-3\delta/8}/4\Big)\ .$$
As $\widetilde{H}_{n,i}$ has the same distribution as the random variable $n^2e^{-S_{i-i_n}}+H^S_{i-i_n} $, then $\p{\widetilde{H}_{n,i} \geq n^{\theta}} = \phi_{i-i_n,2}$ where $\phi_{k,b}$ was previously defined  and estimated in \eqref{phikb}.  So for $A$ large enough \begin{align*}
 \p{ \widetilde{H}_{n,i}\geq n^{\theta}} &\leq Cn^{-\theta(\kappa-1)}.
\end{align*}
Moreover, as $\zeta\leq1$, $\e{\paren{H^S_{j }}^{(\kappa-1)\zeta-\delta/16}}$ is bounded and as $(\kappa-1)\zeta\leq1$, 
\begin{align*}
	\E\Bigg[\Big(\sum_{j \leq  i_n }  H^S_{j }\Big)^{(\kappa-1)\zeta-\delta/16} \Bigg]\leq \sum_{j \leq  i_n }  \e{\paren{H^S_{j }}^{(\kappa-1)\zeta-\delta/16}}\leq C \Ldel\ .
\end{align*}
Markov inequality yields, for any $i\leq\Ldel$,
\begin{align*}
&\P\Big(\sum_{j\leq i_n} H^S_j\geq n^{1-3\delta/8}/4\Big)  \leq C\Ldel n^{- ((\kappa-1)\zeta-\delta/16)(1-3\delta/8)}\leq Cn^{-\delta/16}\ .
\end{align*}
Similarly $\P\big(\sum_{i_n< j\leq i} H^S_j\geq n^{1-3\delta/8}/4 \big) \leq  C  n^{- ((\kappa-1)\zeta-\delta/16)(1-3\delta/8)}\log n   $ and finally by Lemma \ref{compHS}, $\p{H^S_{i_n}\geq n^2} \leq C n^{-2 (\kappa-1)}$. As $\zeta>\theta$, for $\delta$ small enough, this yields
$$p_{5,i} \leq Cn^{- (\kappa-1)\theta-\delta/16}\ .$$
Same kind of arguments give the same bound for $p_{4,i}$. Finally, for $p_{3,i}$,  with the same decomposition as above we can write 
\begin{align*}
p_{3,i} & \leq \p{\widetilde{H}_{n,i}\geq n^{\theta} ,\underline S_{i_n}<-B}  +\p{H^S_{i_n}\geq n^2}+\P\big(  \min_{i_n \leq k \leq i }  S_{k} \leq -B \big) 
\end{align*}
 and by independence,  $\p{\widetilde{H}_{n,i}\geq n^{\theta} ,\underline S_{i_n}<-B}=\p{\widetilde{H}_{n,i}\geq n^{\theta}} \p{\underline S_{i_n}<-B}$. Let  $\lambda=(\kappa-1)\zeta-\delta$. As $\psi(1+ \lambda)<0$, for any $0 \leq l< i$,
 $$\P\Big( \min_{l \leq k \leq i }  S_{k} \leq -B \Big) \leq e^{- \lambda B} \frac{e^{ \psi(1+ \lambda) l}}{1-e^{ \psi(1+ \lambda)}}$$ 
 then
\begin{align*}
p_{3,i} \leq C \paren{e^{- \lambda B} n^{-(\kappa-1) \theta}+ n^{-2 \lambda}+e^{ \psi(1+ \lambda) A\log n }} \leq C e^{- \lambda B} n^{-(\kappa-1) \theta}\ . 
\end{align*}	
\end{proof}

 \subsection{Proof of Theorem \ref{Cortpsdet}}\label{sec3.3}
 
 This Theorem is a direct consequence of Theorem \ref{thm1} and of Remark \ref{localtimeroot}, we only give a proof for the case $\inf_{s \in [0,1]} \psi(s)=0$ with $\psi'(1) < 0$ and $\kappa\in(2,\infty]$ for $\theta>0$ such that $1-\kappa\theta> 1/2-\theta$.
  According to Remark \ref{localtimeroot}, 
for any $\epsilon>0$, 
\begin{align}\label{cvgT}
	\lim_{n\to\infty}\p[*]{\Tn[\lfloor n^{(1+\epsilon)/2}\rfloor]\geq n}=1.
\end{align}
Then fix some $\epsilon>0$ and denote $\theta_\epsilon=2\theta/(1+\epsilon)$. Suppose that $\epsilon$ is small enough to have $2-\kappa\theta_\epsilon\geq 1-\theta_\epsilon$. For $n\geq2$,
\begin{align*}
\p[*]{\frac{\log^+ R^{n}_{n^{\theta}}}{\log n}\geq 1-\kappa\theta+2\epsilon}&\leq  \p[*]{\frac{\log^+  \rangb[\lfloor n^{(1+\epsilon)/2}\rfloor]{n^{\theta}}}{\log n}\geq 1-\kappa\theta+2\epsilon}+\p[*]{\Tn[\lfloor n^{(1+\epsilon)/2}\rfloor]<n}.
\end{align*}
Then, as $\theta=\frac{1+\epsilon}{2}\theta_\epsilon$,
\begin{align*}
\p[*]{\frac{\log^+  \rangb[\lfloor n^{(1+\epsilon)/2}\rfloor]{n^{\theta}}}{\log n}\geq 1-\kappa\theta+2\epsilon}&=\p[*]{\frac{\log^+  \rangb[\lfloor n^{(1+\epsilon)/2}\rfloor]{\paren{n^{(1+\epsilon)/2}}^{\theta_\epsilon}}}{\log n^{(1+\epsilon)/2}}\geq \frac{2(1-\kappa((1+\epsilon)/2)\theta_\epsilon+2\epsilon)}{1+\epsilon}}\\
&\leq\p[*]{\frac{\log^+  \rangb[\lfloor n^{(1+\epsilon)/2}\rfloor]{\lfloor n^{(1+\epsilon)/2}\rfloor^{\theta_\epsilon}}}{\log n^{(1+\epsilon)/2}}\geq 2-\kappa\theta_\epsilon+\frac{2\epsilon}{1+\epsilon}}.
\end{align*}
Then \eqref{cvgT} and Theorem \ref{thm1} show that for any $\epsilon$ small enough, 
$$
\lim_{n\to\infty}\p[*]{\frac{\log^+  R^{n}_{n^{\theta}}}{\log n}\geq (1-\kappa\theta+2\epsilon)}=0\ .
$$
The lower bound and the other cases can be obtained with similar arguments.

  \section{Estimation of the c.d.f.}\label{sec:moments}

 In this section we prove Theorem \ref{thm2.3} and Theorem \ref{thm:ControlEstFModel}. For that purpose, we use the same global strategy as in \cite{DieLer:2017}, that is to say we begin by estimating the moments of $\rho$ (defined in \eqref{defrho}) and we use these estimators to build a family of estimators $\est{F}^\alpha$ of the c.d.f. We finally choose an estimator of the family which minimizes the risk bound.
 The important difference comparing to \cite{DieLer:2017} is that in our case the state space is now a Galton-Watson tree instead of $\Z$. Recall that in this part, we assume not only \ref{ass1} but also \ref{ass2}: 
	\begin{itemize}
	\item the reproduction law of the Galton-Watson is bounded: $\exists K>0,\quad \p{\nu\leq K}=1$ .
	\item given the tree up to generation $n$ and the number of children $\nu_x$ of some $x\in\T$ such that $|x|=n$, the variables $\paren{\omega_{x_i}}_{1\leq i\leq \nu_x}$ are i.i.d. with the same distribution as some variable  $\omega$.
\end{itemize}

  \subsection{Estimation of the moments of $\rho$}

  First remark that the marked tree %(see \cite{Neveu})
  $(x\in\T,\ \NT)$ is a sufficient statistic for the trajectory $\paren{X_t}_{t\in[0,\Tn]}$.
  Indeed for any admissible sequence $\paren{a_k}_{k\in \set{0,\dots, t^{n)}}}$, the likelihood is
 $$
\p{X_0=a_0,\dots,X_{\Tn}=a_{t^{(n)}}\big|\T}=\prod_{x\in\T}\e{\paren{\frac{1}{1+\sum_{i=1}^{\nu_x}e^{-\omega_{x_i}}}}^{n^{(n)}_x}\prod_{i=1}^{\nu_x}\paren{\frac{e^{-\omega_{x_i}}}{1+\sum_{i=1}^{\nu_x}e^{-\omega_{x_i}}}}^{n^{(n)}_{x_i}}\Big|\T}
 $$ 
 where $n_x^{(n)}:= \sum_{k=1} ^{t^{(n)}} \1{a_{k-1}=x^*, \ a_k= x}$. It is then natural to construct our estimator using these random variables.  
  A second important point is that, for a fixed environment and tree, that is to say under $\Pe$,  
 $$
 \forall x\in\T,\quad \Pe \Big({\NT=j\Big|\paren{\NT[y]}_{|y|<|x|}}\Big)=\begin{pmatrix}
\NT[x^*]+j-1\\
 j
 \end{pmatrix} {\rho_x ^i(1-\rho_x)^j}
 $$
 where $\rho_x =\paren{1+e^{-\omega_x}}^{-1}$.
So the moments of $\rho$ are convenient quantities to estimate and we first focus on them. The estimator of $m^{\alpha,\beta}=\e{\rho^\alpha(1-\rho)^\beta}$ is constructed following the same strategy as in \cite{DieLer:2017}.  is based on the simple combinatoric equality: 
\begin{align}\label{eq:combfac1}
\forall i\geq \alpha,\qquad \sum_{j\ge \beta}\binom{i-\alpha+j-\beta}{i-\alpha}a^{i+1}(1-a)^{j}&=a^{\alpha}(1-a)^\beta\enspace.
\end{align}
 Indeed, denoting for any integers $\alpha,\beta,i,j\geq0$,
 \begin{align}\label{eq:Estmalpha}
	 \Phi_{\alpha,\beta}(i,j)=\1{i\geq\alpha+1,j\geq\beta}\frac{\binom{i+j-(\alpha+1+\beta)}{i-(\alpha+1)}}{\binom{i+j-1}{j}}\ ,
\end{align}
we have for any $x\in\T$,
 \begin{align*}
\ee{\Phi_{\alpha,\beta}(\NT[x^*],\NT[x])\Bigg|\paren{\NT[y],\ |y|<|x|}}=&\sum_{j\ge \beta}\Phi_{\alpha,\beta}(\NT[x^*],j)\Pe \Big({\NT=j\Big|\paren{\NT[y],\ |y|<|x|}}\Big)\\
=&\1{\NT[x^*]\geq \alpha +1}\rho_x^{\alpha}(1-\rho_x)^\beta\enspace.
\end{align*}
It is then natural to estimate the moments $ m^{\alpha,\beta}$ of $\rho$ with the following random variable: 
 $$\est[n]{m}^{\alpha,\beta}=\frac{1}{\e{\nu}\rang[\alpha+1]}\sum_{x\in\T}\sum_{\substack{ y,s.t.\\ y^*=x}}\Phi_{\alpha,\beta}(\NT,\NT[y])$$
 where we use the convention: $0/0=0$. Remark that the series here are in fact simple sums as only a finite number of $\NT$ are non zero. These estimators satisfy a concentration property:

 \begin{prop}\label{thm:BorneRisqueGW}
Assume \ref{ass1} and \ref{ass2} and fix $\alpha,\beta\in\Z_+$. There is a positive constant $C$, such that for any integer $n\geq0$ and any real number $z>0$,
 	$$
 	\p{\absj{\est[n]{m}^{\alpha,\beta}-m^{\alpha,\beta}}\geq \frac{\alpha!\beta!}{(\alpha+\beta)!}\frac{K}{\e{\nu}}\sqrt{\frac{z+2\log\rang[\alpha+1]}{2\rang[\alpha+1]}}}\leq Ce^{-z}\enspace.
 	$$
 \end{prop}

\begin{proof} 
	We consider the following construction of the Galton-Watson tree $\T$. Define a sequence of i.i.d.r.v. $\seq[k]{\nu}$ distributed as $\nu$.

We  introduce the sequence $\seqb[k]{y}$ to designate the vertices of the tree $\T\cup\set{e^*}$ in level order, generation by generation. We start with $y_0=e^*$, $y_1=e$ which has $\nu_1$ children. These children are ordered arbitrarily, the first vertex  is called $y_2$ and has $\nu_2$ children, the second one $y_3$ with $\nu_3$ children and so on. Once no more vertices are present at this generation, we move to the next one and so on  (see Figure \ref{fig3}). 
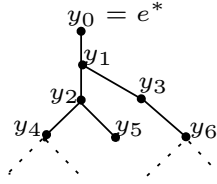
\begin{figure}[ht]
\begin{center}
{\scalebox{1.5} {\input{arbre1.tex} }}
\caption{Vertices numbering } \label{fig3}
\end{center}
\end{figure}
The sequence can be finite with last element $y_{\NM}$. In this case, we start a new tree with a root $y_{\NM+1}$ with $\nu_{\NM+1}$ children and for $k\geq \NM+1$, we put $\NT[y_k]=0$. 
%We denote by $\NM\in\N\cup\set{\infty}$ the number of vertices of $\T$.
Using the sequence $\seq[k]{y}$, the estimators $\est[n]{m}^{\alpha,\beta}$ have the following expression: 
  $$
  \est[n]{m}^{\alpha,\beta}=\frac{1}{\e{\nu}\rang[\alpha+1]}\sum_{k=1}^{\infty}\sum_{\substack{ y,s.t.\\ y^*=y_k}}\Phi_{\alpha,\beta}(\NT[y_k],\NT[y])\ .
  $$ 
  
  Denote by $\seq[k]{\F}$ the following filtration: $\F_0%=\sigma\paren{\NT[y], y^*=e^*}
=\sigma\paren{\NT[e]}$ and for $k\geq 0$, 
\begin{align}\label{defFil}
\F_{k+1}=\sigma \paren{\F_k,\ \sigma\paren{\nuk[k+1],\ (\NT[y], y^*=y_{k+1})}}\ .	
\end{align} 
 
 In the following lemma we introduce a martingale which is the main tool to obtain Proposition \ref{thm:BorneRisqueGW}. 
  \begin{lemm}\label{lem:Martingaleb}  Let $\alpha,\beta\in\Z_+$. For probability measure $\P$, the sequence $\paren{Z^{\alpha,\beta}_{k}}_{k\geq0}$ defined, for all integer $k\geq0$, by
 \begin{align*}
 Z^{\alpha,\beta}_{0}=0\ \text{ and }\ \forall k\geq 1,\ Z^{\alpha,\beta}_{k}=\1{\NT[y_{k}]\geq\alpha+1} \Big(\sum_{y, y^*=y_{k}}\Phi_{\alpha,\beta}\left(\NT[y_{k}],\NT[y]\right)-\e{\nu}m^{\alpha,\beta}\Big)
 \end{align*}
 	is a martingale difference sequence with respect to $\paren{\F_{k}}_{k\in\N}$.
 \end{lemm}
 
 \begin{proof} 
 
  First recall the following result which is a slight variation of Lemma 3.1 in \cite{AidRap}. It can be proved using the same arguments: for any $n\geq1$, under $\P$ the marked tree $\paren{x\in\T,\ \NT}$ is a multi-type Galton-Watson tree. Its initial type is $n$ and its mean matrix is given by: 
 	\begin{align}\label{lemAdR}
 	\forall i,j\geq0,\quad m_{i,j}=\E\Big[\sum_{ |x|=1}\1{\NTi{x}{i}=j}\Big]&=\begin{pmatrix}
                                                                     	i-1+j\\
                                                                     	j
                                                                     \end{pmatrix} \E\Big[\sum_{|x|=1}\frac{e^{-j\omega_{x}}}{\paren{1+e^{-\omega_{x}}}^{i+j}}\Big]\\
                                                                     &=\begin{pmatrix}
                                                                     	i-1+j\\
                                                                     	j
                                                                     \end{pmatrix}\e{\nu}\e{\left(1-\rho\right)^{j} \rho^{i} }\ .\nonumber
 	\end{align}
 
  Now, $Z^{\alpha,\beta}_{k}$ is obviously an integrable and $\F_k$-measurable random variable. Moreover, according to \eqref{lemAdR}, for any $i\geq\alpha+1$ and any $k\geq1$,
 \begin{align*}
 & \E\Big[\1{\NT[y_{k}]=i}\sum_{y, y^*=y_{k}}\Phi_{\alpha,\beta}\left(\NT[y_{k}],\NT[y]\right)\Big|\F_{k-1}\Big]\\
& = \1{\NT[y_{k}]=i}\E\Big [\sum_{|x|=1}\Phi_{\alpha,\beta}\left(i,\NTi{x}{i}\right)\Big] 
	= \1{\NT[y_{k}]=i}\sum_{j=0}^\infty\Phi_{\alpha,\beta}\left(i,j\right)\E\Big [ \sum_{|x|=1}\1{N^{T^i}[x]=j} \Big] \\ 
& = \1{\NT[y_{k}]=i}\sum_{j=\beta}^\infty\binom{i-(\alpha+1)+j-\beta}{i-(\alpha+1)} \e{\nu} \e{\left(1-\rho\right)^{j} \rho^{i} }= \1{\NT[y_{k}]=i}\e{\nu}m^{\alpha,\beta}\ .
\end{align*}
From this last equality we easily obtain $\e{Z^{\alpha,\beta}_{k}\big|\F_{k-1}}=0$ which completes the proof of the lemma.
 \end{proof}

We can now prove Proposition \ref{thm:BorneRisqueGW}. We only have to show that:
 \begin{align}
   	\forall z>0,\ \p{\left|\sum_{k=1}^\infty Z^{\alpha,\beta}_{k}\right|\geq \frac{\alpha !\beta !}{\paren{\alpha+\beta)!}}K\sqrt{\rang[\alpha+1]\paren{z+2\log \rang[\alpha+1]}/2 }}\leq Ce^{-z}\ . \label{eq8}
   \end{align}

 We know, according to Lemma \ref{lem:Martingaleb}, that the process $\seq[k]{M}$ defined by 
 $$M_0=0\ \text{ and }\ \forall k\leq1,\ M_k=\sum_{i=1}^{k} Z^{\alpha,\beta}_{i}$$ 
 is a martingale with respect to $\seq[k]{\F}$.  We could directly apply a concentration inequality but to obtain a better bound we first remark that some of the increments $Z^{\alpha,\beta}_{k}$ are zero so we consider the sequence of stopping times (with respect to the filtration $\seq[k]{\F}$):
 \begin{align}\label{deftau}
 	&\tau_0=0\text{ and }\forall m\geq1,\ \tau_{m+1}=\inf\set{k> \tau_m,\  \NT[y_k]\geq\alpha+1}
 \end{align}
 where $\inf\varnothing=\infty$. As the other variables $Z^{\alpha,\beta}_{k}$ are null, we have that $$M_{\tau_m}=\sum_{i=1}^{m}Z^{\alpha,\beta}_{\tau_i}\ \text{ and }\ \rang[\alpha+1]=\sum_{m=1}^\infty\1{\tau_m<\infty}$$
 where if $\tau_m=\infty, M_\infty=\sum_{i=1}^{\infty} Z^{\alpha,\beta}_{i}$ which is a sum with only a finite number of non-zero terms.
  Now an elementary combinatoric argument shows that, for $i\geq\alpha+1$ and $j\geq\beta$, 
 $$\binom{i-1+j-\alpha-\beta}{i-1-\alpha}\binom{\alpha+\beta}{\alpha}\leq\binom{i-1+j}{i-1}\ .$$
 Thus, for any $i,j\geq0$, 
 $$0\le \Phi_{\alpha,\beta}(i,j)\le \binom{\alpha+\beta}{\alpha}^{-1}=\Phi_{\alpha,\beta}(\alpha+1,\beta)\leq 1$$
 and for any $m\geq0$, $|Z^{\alpha,\beta}_{m}|\leq K$ where $K$ is the upper bound of the support of $\nu$.
 Moreover, for any $m\in\N$, the stopped process $\paren{M_{\tau_m\wedge l}}_{l\geq0}$ is still a martingale. For any $m\in\N$, $\absj{M_{\tau_m\wedge l}}\leq mK$. Therefore the stopped martingale $\paren{M_{\tau_m\wedge l}}_{l\geq0}$ is uniformly integrable and Doob's optional sampling theorem implies that $\left(M_{\tau_m}\right)_{m\in\N}$ is a martingale. As for any $m\geq 0$,
$$
M_{\tau_{m-1}}-\e{\nu}m^{\alpha,\beta}\leq M_{\tau_{m}}\leq M_{\tau_{m-1}}-\e{\nu}m^{\alpha,\beta}+K,
$$
 McDiarmid's inequality (see Theorem 6.7 in \cite{McD89}) leads now to the following concentration result: for any integer $m\geq0$ and any real $z>0$,
 \begin{align*}
 \P\Big(\absj{M_{\tau_m}}\geq \binom{\alpha+\beta}{\alpha}^{-1}K\sqrt{mz/2}\Big)\leq 2e^{-z}\enspace.
 \end{align*}
Then a union bound concludes the proof:
  \begin{align*}
 	&\p{\Big|\sum_{k=1}^\infty Z^{\alpha,\beta}_{k}\Big|\geq \frac{\alpha !\beta !}{\paren{\alpha+\beta)!}}K\sqrt{\rang[\alpha+1]\paren{z+2\log \rang[\alpha+1]}/2}}\\
 =& \p{\bigcup_{m=1}^\infty\set{\rang[\alpha+1]=m\ ;\ \absj{M_{\tau_m}}\geq \binom{\alpha+\beta}{\alpha}^{-1}K\sqrt{\frac{m}{2}\paren{z+2\log{m}}}}}\\
 \leq& \sum_{m=1}^\infty\p{\absj{M_{\tau_m}}\geq \binom{\alpha+\beta}{\alpha}^{-1}K\sqrt{\frac{m}{2}\paren{z+2\log{m}}}} \leq 2e^{-z}\sum_{m\geq1}e^{-2\log{m}}= \frac{\pi^2}{3}e^{-z}\ .
   \end{align*}
 
   \end{proof}
   
   \subsection{Estimation of the cumulative distribution function}\label{sec:cdf}
 In this section we prove Theorem \ref{thm2.3} using the estimation of the moments $m^{\alpha,\beta}$ to approximate the cumulative distribution function $F$ of $\rho$.  
%It is a straightforward consequence of Lemma \ref{lem:biasF} and Lemma \ref{lem:BorneRisqueF}.

\smallskip
 
\noindent Define for any $u\in[0,1]$ and any $\alpha\in\N^*$,
\begin{align*}
F^\alpha(u)=\sum_{k=0}^{\lfloor\alpha u\rfloor-1}\binom{\alpha-1}{k}m^{k,\alpha-1-k}\enspace,
\end{align*}
where $x\to\lfloor x\rfloor$ is the floor function and $\sum_{k=0}^{-1}=0$.
%\textcolor{red}{moins un terme bien choisi pour que la fonction vale 0 en 0 et soit continue.}   
When the function $F$ is regular enough, it is a classical result that it is well approximated by $F^\alpha$. A precise statement is given in Lemma 6 of \cite{DieLer:2017} which we recall here.
\begin{lemm}\label{lem:biasF} Suppose that the function $F\in\mathcal{C}^\gamma$ for some $\gamma\in(0,2]$. For any integer $\alpha\geq 1$,
	$$
	\max_{0\leq \ell\leq \alpha}\ \absj{F(\ell/\alpha) -F^\alpha(\ell/\alpha)}\leq \frac{\|F\|_{\gamma}}{2^\gamma(\alpha+1)^{\gamma/2}}\ .
	$$
\end{lemm}

\medskip

  %According to the results of the previous section, i
It is then natural to use the estimator defined in \eqref{eq:estF}:
\begin{align*}
\est{F}^\alpha(u)=\frac{1}{\rang\e{\nu}}\sum_{x\in\T}\psi^{\lfloor\alpha u\rfloor}_\alpha\paren{\NT[x^*],\NT}\
\textrm{ with }\ \psi^{l}_\alpha(i,j)&=\frac{\1{i\geq \alpha}}{\binom{i-1+j}{\alpha-1}}\sum_{k=0}^{l-1}\binom{i-1}{k}\binom{j}{\alpha-1-k} .
\end{align*}
%We use the conventions $0/0=0$ and $\binom{n}{k}=0$ if $0\leq n< k$. 
Indeed, for any $i\geq \alpha$, $\psi^{l}_\alpha(i,j)=\sum_{k=0}^{l-1}\binom{\alpha-1}{k}\Phi_{k,\alpha-1-k}(i,j),$ where $\Phi_{k,\alpha-1-k}$ is defined in \eqref{eq:Estmalpha}. 
Thus, this estimator is essentially the one obtained from the moment estimators of the previous subsection, but using only the sites $x$ satisfying $\NT\geq \alpha$.  Notice that for any $1\leq l\leq \alpha$ and $i\geq\alpha,j\geq0$,
\begin{align*}
	\sum_{k=0}^{l-1}\binom{i}{k}\binom{j}{\alpha-1-k}&\leq \sum_{k=0}^{\alpha-1}\binom{i-1}{k}\binom{j}{\alpha-1-k}=\binom{i-1+j}{\alpha-1}\enspace,
\end{align*}
thus $\psi^l_\alpha\in[0,1]$. Moreover, Vandermonde's identity:
$$\forall i,j\geq0,\quad \sum_{k=0}^{\alpha-1}\binom{i-1}{k}\binom{j}{\alpha-1-k}=\binom{i-1+j}{\alpha-1}$$
 shows that any $\est{F}^{\alpha}$ is a (random) c.d.f. 
We now have to prove that $\est{F}^\alpha$ estimates correctly $F^\alpha$. This is done in the following lemma.

\begin{lemm}\label{lem:BorneRisqueF}
	For any integers $\alpha,n\geq 1$ and any real $z>0$,
\begin{align*}
	\p{\|\est{F}^\alpha-F^\alpha\|_\infty\geq \frac{K}{\e{\nu}}\sqrt{\frac{z+\log \alpha +2\log \rang}{2\rang}} }\leq Ce^{-z}\ .
\end{align*}
\end{lemm}
\begin{proof}
The proof follows the same lines as the one of Proposition \ref{thm:BorneRisqueGW}.
We introduce the sequence: 
$$
\forall 0\leq \ell\leq \alpha,\ \forall k\in\N,\quad Y^{\alpha,\ell}_{k}= \1{\NT[y_{k}]\geq \alpha}\paren{\sum_{y, y^*=y_{k}}\psi^\ell_\alpha\left(\NT[y_{k}],\NT[y]\right)-\e{\nu}F^\alpha\paren{\ell/\alpha}}\ .
$$
 Using same kind of arguments as in the proof of Lemma \ref{lem:Martingaleb}, we can show that  $\paren{Y^{\alpha,\ell}_{k}}_{k\in\N}$ is a martingale difference sequence with respect to the filtration $\seq[k]{\F}$ defined in \eqref{defFil}. As
 $$\|\est{F}^\alpha-F^\alpha\|_\infty=\max_{1\leq \ell\leq \alpha}\ \absj{\est{F}^\alpha\paren{\ell/\alpha}-F^\alpha\paren{\ell/\alpha}}=\frac{1}{\e{\nu}\rang}\max_{1\leq l\leq \alpha}\ \Big|\sum_{k\in\N}Y^{\alpha,\ell}_{k}\Big|,$$
 Lemma \ref{lem:BorneRisqueF} is equivalent to 
	\begin{align*}
	\forall z>0,\ \p{\max_{1\leq l\leq \alpha}\ \Big|\sum_{k\in\N}Y^{\alpha,\ell}_{k}\Big|\geq K\sqrt{\frac{\rang}{2}(z+\log \alpha+2\log \rang)}}\leq Ce^{-z}\ .
	\end{align*}
So we now consider martingale $\paren{M^{\alpha,l}_{k}}_{k\in\N}$ defined by $M^{\alpha,l}_{k}:=\sum_{j\leq k}Y^{\alpha,l}_{j}$ and the same sequence of stopping times $\seq[m]{\tau}$ as the one defined in \eqref{deftau} except that $\alpha+1$ is replaced by $\alpha$:
\begin{align*}
 	&\tau_0=0\text{ and }\forall m\geq1,\ \tau_{m+1}=\inf\set{l> \tau_m,\  \NT[y_l]\geq\alpha}\ .
 \end{align*}
The process $\paren{M^{\alpha,l}_{\tau_m}}_{m\in\N}$ is still a martingale and for any $m\geq1$, 
$$  
M^{\alpha,l}_{\tau_{m-1}}-\e{\nu}F^\alpha\paren{\ell/\alpha}\leq M^{\alpha,l}_{\tau_m}\leq M^{\alpha,l}_{\tau_{m-1}}-\e{\nu}F^\alpha\paren{\ell/\alpha}+K\ .
$$
Then Mc Diarmid's inequality (Theorem 6.7 in \cite{McD89}) yields, for any $1\leq \ell\leq \alpha$,
  \begin{align*}%\label{eq:randterm}
 	\forall k\in\N,\ \forall z>0,\quad\p{\absj{M^{\alpha,\ell}_{k}}\geq K\sqrt{\frac{kz}{2}}}\leq 2e^{-z}
 \end{align*}
and a union bound gives the result:
  \begin{align*}
 	&\P\Bigg(\max_{1\leq \ell\leq \alpha}\ \Big|\sum_{k\in\N}Y^{\alpha,\ell}_{k}\Big|\geq K\sqrt{\frac{\rang}{2}(z+\log \alpha+2\log \rang)}\Bigg)\\
 =& \p{\bigcup_{k=1}^\infty\Bigg\{\rang=k\ ;\ \max_{1\leq \ell\leq \alpha}\absj{M^{\alpha,\ell}_{k}}\geq K\sqrt{\frac{k}{2}(z+\log \alpha+2\log{k})}\Bigg\}}\\
 \leq& \sum_{k=1}^\infty\sum_{\ell=1}^\alpha\p{\absj{M^{\alpha,\ell}_{k}}\geq K\sqrt{\frac{k}{2}(z+\log \alpha+2\log{k})}}  \leq 2e^{-z}\sum_{k\geq1}\sum_{\ell=1}^\alpha\frac{e^{-2\log{k}}}{\alpha}= \frac{\pi^2}{3}e^{-z}\ .
   \end{align*}
\end{proof}

The inequality of Theorem \ref{thm2.3} follows now immediately from Lemma \ref{lem:biasF}, Lemma \ref{lem:BorneRisqueF} and the fact that $F$ is $\gamma\wedge1$-H\"older.

\subsection{Proof of Theorem \ref{thm:ControlEstFModel}}

The main step of the proof is to show that we can find a constant $C$ such that for any integer $n\geq 1$, and any real $z>0$, there exists a r.v. $\widehat{\alpha}_{z,n}$ depending only on $\paren{X_{k}}_{0\leq k\leq \Tn}$  such that
	\begin{align}\label{ResInt}
	\p{\|\est{F}^{\widehat{\alpha}_{z,n}}-F\|_\infty\geq \inf_{\alpha\geq1}\set{\frac{4K}{\e{\nu}}\sqrt{\frac{z+\log \alpha +2\log \rang}{2\rang}}+\frac{10\|F\|_\gamma}{\alpha^{\gamma/2}}}}\leq Ce^{-z}\ .
        \end{align}
Theorem \ref{thm:ControlEstFModel} can then be easily deduced from \eqref{ResInt} using the estimations of $\rang$ given in Theorem \ref{thm1}.

\noindent Equation \eqref{ResInt} and the estimator $\est{F}^{\widehat{\alpha}_{z,n}}$ is obtained from the collection $\paren{\est{F}^{\alpha}}_{\alpha\geq1}$ via Goldenshluger-Lepski's method (see \cite{GolLep2008}). We explain the construction below. \\
We fix first some real number $z>0$ and define for any integer $\alpha\geq 1$, 
\[\GL=\frac{K}{\e{\nu}}\sqrt{\frac{z+\log \alpha +2\log \rang}{2\rang}}\ \text{ and }\ \Delta(\alpha)=\sup_{\alpha'\ge 1}\set{\norm{\est{F}^{\alpha'}-\est{F}^{\alpha\wedge \alpha'}}-\GL[\alpha']}\enspace.\]
\lemref{BorneRisqueF} and a union bound show that, for some $C>0$,
\begin{equation}\label{eq:ContralVarUnionBound}
\P\paren{\forall \alpha\ge 1,\ \norm{\est{F}^\alpha-F^\alpha}_{\infty}\leq \GL}\ge 1-Ce^{-z}\enspace. 
\end{equation}
Moreover, by \lemref{biasF} and the fact that $F$ is $\gamma\wedge1$-H\"older, we have the following control on the bias of the estimator:
\begin{equation}\label{eq:ControlBiasF}
\norm{F-F^\alpha}_{\infty}\le \frac{2\|F\|_\gamma}{\alpha^{\gamma/2}}\enspace. 
\end{equation}

The random variable $\widehat{\alpha}_{z,n}$ is defined by
\[\widehat{\alpha}_{z,n}=\arg\min_{\alpha\ge 1}\set{\Delta(\alpha)+\GL}\enspace.\]
We now have to check that $\widehat{\alpha}_{z,n}$ satisfies \eqref{ResInt}. Let 
\[\Omega=\set{\forall \alpha\ge 1,\ \norm{\est{F}^\alpha-F^\alpha}_{\infty}\leq \GL}\enspace.\]
By \eqref{eq:ContralVarUnionBound}, $\P\paren{\Omega} \ge 1-Ce^{-z}$. 
On $\Omega$, using first the definition of $\Delta(\alpha)$ and then the one of  $\widehat{\alpha}_{z,n}$, we get for any $\alpha\geq1$,
\begin{align*}
 \norm{\est{F}^{\widehat{\alpha}_{z,n}}-F}_{\infty}&\le \norm{\est{F}^{\widehat{\alpha}_{z,n}}-\est{F}^{\widehat{\alpha}_{z,n}\wedge \alpha}}_{\infty}+\norm{\est{F}^{\alpha}-\est{F}^{\widehat{\alpha}_{z,n}\wedge \alpha}}_{\infty}+\norm{\est{F}^{\alpha}-F^\alpha}_{\infty}+\norm{F^{\alpha}-F}_{\infty}\\
 &\le \Delta(\alpha)+\GL[\widehat{\alpha}_{z,n}]+\Delta(\widehat{\alpha}_{z,n})+\GL+\GL+\frac{2\|F\|_\gamma}{\alpha^{\gamma/2}}\\
 &\le 2\Delta(\alpha)+3\GL+\frac{2\|F\|_\gamma}{\alpha^{\gamma/2}}\enspace.
\end{align*}
Now, for any $\alpha'\geq \alpha$,
\begin{align*}
\norm{\est{F}^{\alpha'}-\est{F}^{\alpha}}_{\infty}-\GL[\alpha']&\le\paren{\norm{\est{F}^{\alpha'}-F^{\alpha'}}_{\infty}-\GL[\alpha']}+\norm{F^{\alpha'}-F^{\alpha}}_{\infty}+\norm{\est{F}^{\alpha}-F^\alpha}_{\infty}\enspace.
\end{align*}
The first term is non positive on $\Omega$, the second one is bounded by $\frac{4\|F\|_\gamma}{\alpha^{\gamma/2}}$ by \eqref{eq:ControlBiasF} and the third one is smaller than $\GL[\alpha]$. It follows that, on $\Omega$, for any $\alpha\geq1$,

$$
\Delta(\alpha)\le \frac{4\|F\|_\gamma}{\alpha^{\gamma/2}}+\GL[\alpha]\quad \text{ and finally }\quad \norm{\est{F}^{\widehat{\alpha}_{z,n}}-F}_{\infty}\le 4\GL+\frac{10\|F\|_\gamma}{\alpha^{\gamma/2}}\enspace.
$$
This proves \eqref{ResInt}. \\

\noindent {\bf Acknowledgment}: we are grateful to an anonymous referee for comments that were useful to  improve the presentation of the paper.

\bibliographystyle{alpha}
\bibliography{thbiblio}

\end{document}

%% file: RecCrit.tex
\begin{picture}(0,0)%
\includegraphics{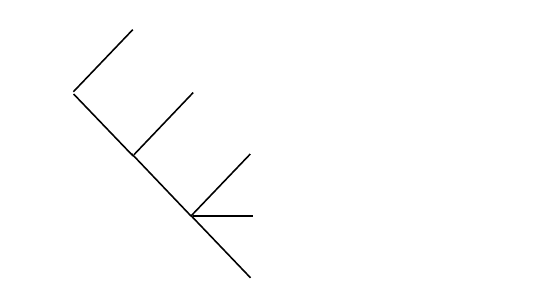}%
\end{picture}%
\setlength{\unitlength}{3947sp}%
\begingroup\makeatletter\ifx\SetFigFont\undefined%
\gdef\SetFigFont#1#2#3#4#5{%
  \reset@font\fontsize{#1}{#2pt}%
  \fontfamily{#3}\fontseries{#4}\fontshape{#5}%
  \selectfont}%
\fi\endgroup%
\begin{picture}(2611,1416)(603,-3705)
\put(1838,-3038){\makebox(0,0)[lb]{\smash{{\SetFigFont{5}{6.0}{\rmdefault}{\mddefault}{\updefault}$>0$ Positive recurrent}}}}
\put(1382,-3174){\makebox(0,0)[lb]{\smash{{\SetFigFont{5}{6.0}{\rmdefault}{\mddefault}{\updefault}$=0$}}}}
\put(1258,-3389){\makebox(0,0)[lb]{\smash{{\SetFigFont{5}{6.0}{\rmdefault}{\mddefault}{\updefault}$\psi'(1)$}}}}
\put(618,-2746){\makebox(0,0)[lb]{\smash{{\SetFigFont{5}{6.0}{\rmdefault}{\mddefault}{\updefault}$\underset{[0,1]}{\inf}\psi$}}}}
\put(1114,-2606){\makebox(0,0)[lb]{\smash{{\SetFigFont{5}{6.0}{\rmdefault}{\mddefault}{\updefault}$>0$}}}}
\put(1115,-2890){\makebox(0,0)[lb]{\smash{{\SetFigFont{5}{6.0}{\rmdefault}{\mddefault}{\updefault}$\leq 0$}}}}
\put(1256,-2439){\makebox(0,0)[lb]{\smash{{\SetFigFont{5}{6.0}{\rmdefault}{\mddefault}{\updefault}Transient}}}}
\put(1538,-2735){\makebox(0,0)[lb]{\smash{{\SetFigFont{5}{6.0}{\rmdefault}{\mddefault}{\updefault}Positive Recurrent}}}}
\put(1377,-2914){\makebox(0,0)[lb]{\smash{{\SetFigFont{5}{6.0}{\rmdefault}{\mddefault}{\updefault}$<0$}}}}
\put(1852,-3341){\makebox(0,0)[lb]{\smash{{\SetFigFont{5}{6.0}{\rmdefault}{\mddefault}{\updefault}$=0$  Null recurrent}}}}
\put(1841,-3650){\makebox(0,0)[lb]{\smash{{\SetFigFont{5}{6.0}{\rmdefault}{\mddefault}{\updefault}$<0$ Null recurrent}}}}
\end{picture}%

%% file: Psi.tex
\begin{picture}(0,0)%
\includegraphics{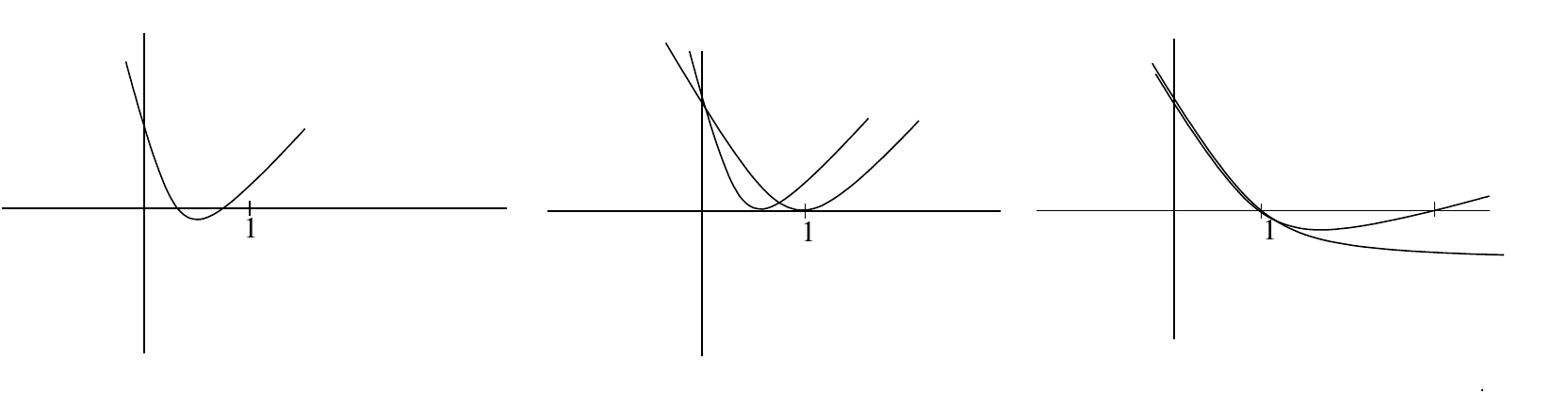}%
\end{picture}%
\setlength{\unitlength}{3947sp}%
\begingroup\makeatletter\ifx\SetFigFont\undefined%
\gdef\SetFigFont#1#2#3#4#5{%
  \reset@font\fontsize{#1}{#2pt}%
  \fontfamily{#3}\fontseries{#4}\fontshape{#5}%
  \selectfont}%
\fi\endgroup%
\begin{picture}(7976,1997)(-11,-2713)
\put(6151,-2386){\makebox(0,0)[lb]{\smash{{\SetFigFont{8}{9.6}{\rmdefault}{\mddefault}{\updefault}{\color[rgb]{0,0,0}C : sub-diffusive  and diffusive case}%
}}}}
\put(809,-851){\makebox(0,0)[lb]{\smash{{\SetFigFont{9}{10.8}{\rmdefault}{\mddefault}{\updefault}$\psi(t)$}}}}
\put(5939,-871){\makebox(0,0)[lb]{\smash{{\SetFigFont{9}{10.8}{\rmdefault}{\mddefault}{\updefault}$\psi(t)$}}}}
\put(3524,-896){\makebox(0,0)[lb]{\smash{{\SetFigFont{9}{10.8}{\rmdefault}{\mddefault}{\updefault}$\psi(t)$}}}}
\put(5998,-1900){\makebox(0,0)[lb]{\smash{{\SetFigFont{8}{9.6}{\rmdefault}{\mddefault}{\updefault}$0$}}}}
\put(3577,-1900){\makebox(0,0)[lb]{\smash{{\SetFigFont{8}{9.6}{\rmdefault}{\mddefault}{\updefault}$0$}}}}
\put(739,-1897){\makebox(0,0)[lb]{\smash{{\SetFigFont{8}{9.6}{\rmdefault}{\mddefault}{\updefault}$0$}}}}
\put(6996,-1736){\makebox(0,0)[lb]{\smash{{\SetFigFont{9}{10.8}{\rmdefault}{\mddefault}{\updefault}$\kappa<+\infty$}}}}
\put(7647,-2057){\makebox(0,0)[lb]{\smash{{\SetFigFont{9}{10.8}{\rmdefault}{\mddefault}{\updefault}$\kappa=\infty$}}}}
\put(901,-2386){\makebox(0,0)[lb]{\smash{{\SetFigFont{8}{9.6}{\rmdefault}{\mddefault}{\updefault}{\color[rgb]{0,0,0}A : very slow case}%
}}}}
\put(3751,-2386){\makebox(0,0)[lb]{\smash{{\SetFigFont{8}{9.6}{\rmdefault}{\mddefault}{\updefault}{\color[rgb]{0,0,0}B : slow cases}%
}}}}
\end{picture}%

%% file: arbre1.tex
\begin{picture}(0,0)%
\includegraphics{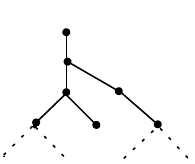}%
\end{picture}%
\setlength{\unitlength}{3947sp}%
\begingroup\makeatletter\ifx\SetFigFont\undefined%
\gdef\SetFigFont#1#2#3#4#5{%
  \reset@font\fontsize{#1}{#2pt}%
  \fontfamily{#3}\fontseries{#4}\fontshape{#5}%
  \selectfont}%
\fi\endgroup%
\begin{picture}(922,769)(5832,-1007)
\put(6389,-616){\makebox(0,0)[lb]{\smash{{\SetFigFont{6}{7.2}{\rmdefault}{\mddefault}{\updefault}$y_3$}}}}
\put(6089,-337){\makebox(0,0)[lb]{\smash{{\SetFigFont{6}{7.2}{\rmdefault}{\mddefault}{\updefault}$y_0=e^*$}}}}
\put(6158,-499){\makebox(0,0)[lb]{\smash{{\SetFigFont{6}{7.2}{\rmdefault}{\mddefault}{\updefault}$y_1$}}}}
\put(6017,-661){\makebox(0,0)[lb]{\smash{{\SetFigFont{6}{7.2}{\rmdefault}{\mddefault}{\updefault}$y_2$}}}}
\put(6602,-805){\makebox(0,0)[lb]{\smash{{\SetFigFont{6}{7.2}{\rmdefault}{\mddefault}{\updefault}$y_6$}}}}
\put(5864,-796){\makebox(0,0)[lb]{\smash{{\SetFigFont{6}{7.2}{\rmdefault}{\mddefault}{\updefault}$y_4$}}}}
\put(6290,-796){\makebox(0,0)[lb]{\smash{{\SetFigFont{6}{7.2}{\rmdefault}{\mddefault}{\updefault}$y_5$}}}}
\end{picture}%